\documentclass[11pt]{amsart}
\usepackage{amsmath}
\usepackage[psamsfonts]{amssymb}
\usepackage[all]{xy}
\usepackage{graphicx}
\usepackage[T1]{fontenc}
\usepackage[bitstream-charter]{mathdesign}
\usepackage{hyperref}
\usepackage{verbatim}
\usepackage[english]{babel}
\usepackage{color}
\theoremstyle{definition}
\newtheorem{theorem}{Theorem}[section]
\newtheorem{prop}[theorem]{Proposition}
\newtheorem{lemma}[theorem]{Lemma}
\newtheorem{cor}[theorem]{Corollary}
\newtheorem{assm}[theorem]{Assumptions}

\newtheorem{question}[theorem]{Question}

\newtheorem{ex}[theorem]{Example}
\theoremstyle{remark}
\newtheorem{dfn}[theorem]{Definition}
\newtheorem{remark}[theorem]{Remark}
\newtheorem{claim}[theorem]{Claim}
\numberwithin{equation}{section}
\def\nsubset{\not\subset}
\def\co{\colon\thinspace}

\def\ep{\epsilon}

\def\R{\mathbb{R}}
\def\Z{\mathbb{Z}}
\def\N{\mathbb{N}}

\title{Local rigidity, contact homeomorphisms, and conformal factors}
\author{Michael Usher}
\address{Department of Mathematics\\University of Georgia\\Athens, GA 30602}
\email{usher@uga.edu}
\begin{document}
\begin{abstract}
We show that if the image of a Legendrian submanifold under a contact homeomorphism (\emph{i.e.} a homeomorphism that is a $C^0$-limit of contactomorphisms) is smooth then it is Legendrian, assuming only positive local lower bounds on the conformal factors of the approximating contactomorphisms.  More generally the analogous result holds for coisotropic submanifolds in the sense of \cite{Huang}.   This is a contact version of the Humili\`ere-Leclercq-Seyfaddini coisotropic rigidity theorem in $C^0$ symplectic geometry, and the proof adapts the author's recent re-proof of that result in \cite{U3} based on a notion of local rigidity of points on locally closed subsets.  We also provide two different flavors of examples showing that a contact homeomorphism can map a submanifold that is transverse to the contact structure to one that is smooth and tangent to the contact structure at a point.
\end{abstract}
\maketitle

\section{Introduction}
The Eliashberg-Gromov symplectic rigidity theorem, stating that the group of symplectic diffeomorphisms of a symplectic manifold is $C^0$-closed in the group of all diffeomorphisms, led to the notion of a symplectic homeomorphism as a homeomorphism which is a $C^0$-limit of symplectic diffeomorphisms, and to the field of ``$C^0$ symplectic topology,'' studying the properties of symplectic manifolds that are invariant under symplectic homeomorphisms.  Analogous ideas in the setting of contact manifolds have only fairly recently begun to be developed.  In particular \cite{MSrig}, following older ideas of Eliashberg, gave the first full proof in the literature of the contact version of the Eliashberg-Gromov theorem, and \cite{M19} gave an alternative proof based on a characterization of contact diffeomorphisms in terms of their effect on a version of Eliashberg's shape invariant. 

Throughout this paper, a \textbf{contact homeomorphism} of a contact manifold $(Y,\xi)$ is by definition a homeomorphism of $Y$ that arises as limit of some sequence of contact diffeomorphisms with respect to the $C^0$ (compact-open) topology (this differs from the usage in \cite{MSp}, in which the contact homeomorphism group is a certain subgroup of the group of topological automorphisms mentioned at the end of this paragraph, and thus is significantly smaller than what we define as the contact homeomorphism group).   If $\xi$ is cooriented, say with $\xi=\ker\alpha$ where $\alpha\in \Omega^1(Y)$, it is well-known  that questions about contact diffeomorphisms if $(Y,\xi)$ can be converted to questions about symplectic diffeomorphisms of the symplectization $(\R\times Y,d(e^r\alpha))$ (where $r$ is the coordinate on $\R$).  Specifically, a diffeomorphism $\psi\co Y\to Y$ obeys $\psi^*\alpha=f\alpha$ for a smooth function $f\co Y\to (0,\infty)$ if and only if the diffeomorphism $\Psi_f\co \R\times Y\to\R\times Y$ defined by $\Psi_f(r,y)=\left(r-\log f(y),\psi(y)\right)$ is a symplectomorphism.  An analogous device is \emph{not} in general available for contact homeomorphisms, essentially because of the dependence of $\Psi_f$ above on the conformal factor $f$, which in turn depends on the derivative of $\psi$.  It is quite possible for a sequence $\{\psi_m\}$ of contactomorphisms to $C^0$-converge to a homeomorphism $\psi$ while the logarithms of the conformal factors $f_m$ given by $\psi_{m}^{*}\alpha=f_m\alpha$ are unbounded (see Section \ref{accel} for one family of examples), in which case the $\Psi_{f_m}$  do not converge.  On the other hand, in \cite{MSp} the authors consider the more restricted class of ``topological automorphisms'' of a contact manifold, defined to be limits $C^0$-limits of sequences of contactomorphisms $\{\psi_m\}_{m=1}^{\infty}$ with the property that the corresponding conformal factors $f_m$ converge uniformly. 

The main question motivating this paper is the following:

\begin{question}\label{mainq}
Let $(Y,\xi)$ be a contact manifold and $\psi\co Y\to Y$ a contact homeomorphism. Suppose that $\Lambda\subset Y$ is a Legendrian submanifold such that $\psi(\Lambda)$ is a smooth submanifold.  Must $\psi(\Lambda)$ be Legendrian?
\end{question}

More generally we will consider the situation where $\Lambda$ is coisotropic in the sense of Definition \ref{coisodef} (this is the same definition used in \cite{Huang}); under this definition, Legendrian submanifolds are precisely the coisotropic submanifolds of dimension $\frac{1}{2}(\dim Y-1)$.  In the $C^0$ symplectic world, the main result of \cite{HLS} asserts that the image under a symplectic homeomorphism of a coisotropic submanifold of a symplectic manifold is coisotropic provided that is is smooth.  In \cite[Theorem 1.3]{RZ} this is used to deduce an affirmative answer to Question \ref{mainq} (and also its analogue for coisotropic submanifolds) in the special case that $\psi$ is a topological automorphism in the sense of \cite{MSp}.  

However the authors of \cite{RZ} express doubt (in their Remark 4.4) that the same conclusion continues to hold when one considers fully general contact homeomorphisms.  To indicate why one indeed should not blithely assume that obvious analogues of $C^0$ symplectic results hold in the $C^0$ contact context, note that \cite[Theorem 2]{LS} shows that a smooth embedding of a compact $n$-dimensional manifold into $\R^{2n}$ that is a $C^0$-limit of Lagrangian embeddings is itself Lagrangian, whereas \emph{any} smooth embedding of an $n$-dimensional manifold into a $(2n+1)$-dimensional contact manifold that satisfies a mild homotopy-theoretic hypothesis can be $C^0$-approximated by Legendrian embeddings (see \cite[Theorem 2.5]{Et} if $n=1$ and \cite[Theorem 7.25]{CE} if $n>1$).  It is not clear from the proofs of the latter results whether the approximating Legendrian embeddings can be arranged to be the restrictions of a uniformly convergent sequence of contactomorphisms.

While we do not resolve Question \ref{mainq} here, we do give an affirmative answer under a significantly weaker hypothesis on the conformal factors of the approximating sequence $\{\psi_m\}$ for the contact homeomorphism $\psi$ than in \cite[Theorem 1.3]{RZ} (which required these conformal factors to converge uniformly).  Specifically we will require $\psi$ to be \textbf{bounded below near} all points of $C$ in the sense of Definition \ref{bddef}; for example if $\alpha$ is a contact form for $\xi$ and if the approximating sequence $\psi_m$ has $\psi_{m}^{*}\alpha=f_m\alpha$ this will hold if the functions $|f_m|$ satisfy $m$-independent positive lower bounds on some neighborhood of $C$.  Our main result is then:

\begin{theorem}\label{main} Suppose that $(Y,\xi)$ is a contact manifold and $\psi\co Y\to Y$ is a contact homeomorphism.  If $C$ is a coisotropic submanifold of $Y$ such that $\psi(C)$ is smooth and such that $\psi$ is bounded below near every point of $C$, then $\psi(C)$ is coisotropic.
\end{theorem}


\begin{remark}
A contact form $\alpha$ on a $(2n+1)$-dimensional manifold $Y$ induces a Borel measure $\mu_{\alpha}$ on $Y$ by setting $\mu_{\alpha}(E)=\int_{E}\alpha\wedge (d\alpha)^{\wedge n}$.  If $\psi\co Y\to Y$ is a contactomorphism, say with $\psi^{*}\alpha=f\alpha$, then evidently one has \[ \mu_{\alpha}(\psi(E))=\int_{E}|f|^{n+1}\alpha\wedge (d\alpha)^{\wedge n}\] for all Borel sets $E$.  Thus imposing local lower bounds on the absolute values of the functions $f_m$ given by $\psi_{m}^{*}\alpha=f_m\alpha$ for an approximating sequence $\psi_m$ for a contact homeomorphism $\psi$ amounts to imposing  lower bounds on the ratios $\frac{\mu_{\alpha}(\psi_m(U))}{\mu_{\alpha}(U)}$ for appropriate open sets $U$, or equivalently local lower bounds on the Jacobian determinants of the $\psi_m$ when these are expressed in local coordinates.  

If $\{\psi_m\}_{m=1}^{\infty}$ is any sequence of contactomorphisms such that both $\psi_m$ and $\psi_{m}^{-1}$ obey uniform local bounds on their Lipschitz constants, then the Arzel\`a-Ascoli theorem implies that some subsequence of $\{\psi_m\}$ converges in the compact-open topology to a contact homeomorphism $\psi$, and then both $\psi$ and $\psi^{-1}$ will be bounded below near every point.
\end{remark}

Neither the hypothesis nor the conclusion of Theorem \ref{main} is manifestly preserved under replacing $\psi$ by $\psi^{-1}$; rather, applying Theorem \ref{main} to $\psi^{-1}$ leads to the statement that if $\psi$ is a contact homeomorphism that is bounded \emph{above} near all points of $N$ and if $N$ and $\psi(N)$ are both smooth submanifolds with $N$ \emph{not} coisotropic then $\psi(N)$ is also not coisotropic.  The following theorem, proven in Section \ref{exsect} shows however that the situation is different if instead of asking whether the whole image $\psi(N)$ is coisotropic one just asks whether it is coisotropic at an isolated point. 

\begin{theorem}\label{exprop}
For any contact manifold  $(Y,\xi)$  of dimension $2n+1\geq 3$, there exist contact homeomorphisms $\psi\co Y\to Y$ and smooth $n$-dimensional submanifolds $\Lambda\subset Y$ such that $\psi(\Lambda)$ is a smooth submanifold and, for some point $p\in \Lambda$, we have $T_p\Lambda\nsubset\xi_p$ but $T_{\psi(p)}\psi(\Lambda)\subset \xi_p$. In fact, $\psi$ can be chosen to have any of the following properties:
\begin{itemize}\item[(i)]  $\psi$ is bounded both above and below near $p$; or
\item[(ii)] $\psi$ is bounded below but not above near $p$; or
\item[(iii)] $\psi$ restricts to $\Lambda$ as a smooth map, and is bounded above but not below near $p$. \end{itemize}
\end{theorem}

\begin{remark}
In \cite[Section 5.3.3]{Mas}, Massot suggests an alternative definition of a contact homeomorphism of a compact contact manifold $(Y,\xi)$ as  a homeomorphism of $Y$ that is bi-Lipschitz with respect to one and hence any Carnot-Carath\'eodory distance induced by $\xi$ (declaring the distance between two points to be the infimal length of a Legendrian arc connecting them, as measured by an auxiliary Riemannian metric).  It is noted on \cite[p. 89]{Mas} that this is probably a different notion than the one based on $C^0$-limits that we use in this paper, and the examples in Section \ref{accel} (which are the ones that we use to prove variation (iii) of Theorem \ref{exprop}) confirm this expectation.  Indeed, restricting to the three-dimensional case for ease of notation, these contact homeomorphisms $\psi$ are given in the hypersurface $\{y=0\}$ within a Darboux cube $((-1,1)^3,\ker(dz-ydx))$ by $\psi(x,0,z)=(x,0,g(z))$ where the function $g\co \R\to\R$, a general formula for appears in Proposition \ref{wall}, typically has $g'(0)=0$ and can even be arranged to vanish to infinite order at $0$ as in Example \ref{fourinf}.  In particular $\psi$ preserves the $z$ axis and restricts to it as a non-bi-Lipschitz function (with respect either to the standard metric or the restriction of the Carnot-Carath\'eodory distance). \end{remark}


\subsection{Outline of the paper}

The proof of Theorem \ref{main} is a contact version of the author's recent re-proof in \cite{U3} of the Humili\`ere-Leclercq-Seyfaddini theorem \cite{HLS} on $C^0$-rigidity of coisotropic submanifolds of symplectic manifolds.  As in \cite{U3}, the plan is to characterize coisotropic submanifolds in terms of a notion that we call ``local rigidity'' and prove that this notion is invariant under the appropriate class of homeomorphisms.  The difference in length between Sections \ref{rigsect} through \ref{coisosect} of this paper and \cite[Sections 1 and 2]{U3} is explained by a combination of the contact geometric case being objectively more complicated and the theory of coisotropic submanifolds in contact geometry being less well-developed than that of their  counterparts in symplectic geometry.\footnote{One indication of this underdevelopment is that the very recent sources \cite{RZ},\cite{LdL}, \cite{M19} all have conflicting definitions of a coisotropic submanifold of a contact manifold; as mentioned earlier our definition is that used in \cite{Huang} and \cite{RZ}.}  

Section \ref{rigsect} defines our notion of a point $p$ on a locally closed subset $N$ of a contact manifold being locally rigid with respect to $N$.  The symplectic version of this from \cite{U3} was in terms of the Hofer energy needed to locally disjoin arbitrarily small neighborhoods of $p$ from $N$, and the contact version introduced here is essentially the same but with the Shelukhin norm \cite{She} on the identity component of the contactomorphism group (defined using absolute values of contact Hamiltonians) used in place of the Hofer norm.   A crucial fact about local rigidity  is then Proposition \ref{rigpres}, asserting that if $p$ is locally rigid with respect to $N$ and if $\psi$ is a contact homeomorphism of $Y$ that is bounded below near $p$ in the sense of Definition \ref{bddef}, then $\psi(p)$ is locally rigid with respect to $\psi(N)$.  The need for the boundedness hypothesis can be understood in terms of the fact that the Shelukhin norm $\|\cdot\|$ (unlike the Hofer norm in symplectic topology) is not conjugation-invariant; rather $\|\psi\circ\phi\circ\psi^{-1}\|$ can be bounded in terms of $\|\phi\|$ and the conformal factor of $\psi$.

Section \ref{legrigsect} proves Corollary \ref{leglr}, asserting that points on Legendrian submanifolds $\Lambda$ are always locally rigid; this is the only point in the paper that depends on pseudoholomorphic curve techniques.  To prove it we show in Theorem \ref{hypermain} that, under suitable assumptions, there is a positive lower bound on the Shelukhin norm of a contactomorphism that disjoins a given pre-Lagrangian submanifold from $\Lambda$; this follows from Lemma \ref{rs} which establishes a lower bound for the Hofer norm of a symplectomorphism of the symplectization that disjoins a compact Lagrangian submanifold from $\R\times \Lambda$, using a number of technical ingredients from \cite{RS} and references therein.  Lemma \ref{rs} requires a rather restrictive hypothesis---hypertightness in the sense of Definition \ref{hyperdef}---on $\Lambda$, but because our definition of local rigidity is indeed local we can use tubular neighborhood theorems to deduce relevant information from Lemma \ref{rs} about any Legendrian submanifold, even one that is not closed as a subset.  We also observe in Corollary \ref{htr} that Theorem \ref{hypermain} implies, in the special case of hypertight Legendrians, \cite[Conjecture 1.10]{RZ} on the contact analogue of Chekanov-Hofer pseudometrics on orbits of submanifolds.  

The proof of Theorem \ref{main} is completed at the end of Section \ref{coisosect}, in which we characterize coisotropic submanifolds in terms of local rigidity.  Proposition \ref{nonco} shows that a point $p$ on a submanifold $C$ is locally rigid only if $C$ is coisotropic at $p$; this follows by a variation on arguments from \cite{U1},\cite{RZ}. Unlike in the symplectic case (see \cite[Theorem 2.1]{U3}) it is not known to the author whether the converse to this holds, except in the case that $C$ is Legendrian in which case the converse is already given by Corollary \ref{leglr}.  The reason is that, on non-Legendrian coisotropic submanifolds $C$ of contact manifolds $(Y,\xi)$, there are two fundamentally different types of points $p$: those for which $T_pC\subset \xi_p$, and those at which $C$ is transverse to $\xi$.  However, as we show in Corollary \ref{coleg}, points of the latter type form an open dense subset of $C$, and moreover any such point is contained in a Legendrian submanifold that is in turn contained in $C$.  (The behavior of $C$ near those points $p$ where $T_pC\subset \xi_p$ can, on the other hand, be quite complicated, cf. \cite{Huang}.)  Given Corollary \ref{leglr} and Proposition \ref{nonco}, it then follows that the coisotropic submanifolds are precisely those submanifolds of a contact manifold for which an open and dense subset of the points are locally rigid; Proposition \ref{rigpres} proves that this property is preserved under contact homeomorphisms that are bounded below near every point of the submanifold, thus proving Theorem \ref{main}.

The final Section \ref{exsect} explains the examples referenced in Theorem \ref{exprop}, whose proof is completed at the very end of the paper.  One of these constructions (see Section \ref{bosect}) is obtained by a straightforward modification of a construction from \cite[Section 4]{BO}; this yields a contact homeomorphism $\psi$ that is bounded both above and below, which maps a codimension-two contact submanifold $Z$ (the locus where $x_1=y_1=0$ in the notation of Proposition \ref{boprop}) to an explicit non-contact submanifold, though the behavior of the contact homeomorphism away from this contact submanifold seems difficult to understand.  By restricting to submanifolds of $Z$ one obtains in Corollary \ref{bocor} the examples indicated in item (i) of Theorem \ref{exprop}, as well as similar examples which, instead of being Legendrian, are coisotropic of some codimension smaller than $n+1$.   The other construction (in Section \ref{accel}) is perhaps more distinctively contact-geometric, and uses the flow of an explicit time-dependent contact Hamiltonian vector field on the complement of a Legendrian torus $T$ that extends continuously over the torus and whose flow contracts small neighborhoods of $T$ by increasingly large factors as one approaches $T$. This construction, unlike the other one, leads to $\psi|_{\Lambda}$ being a smooth map (not just to $\psi(\Lambda)$ being a smooth submanifold). In fact the approximating sequence $\psi_m$ to $\psi$ has the property that, where $\Lambda$ is as in Proposition \ref{exprop}, $\psi_m|_{\Lambda}$ converges to $\psi|_{\Lambda}$ in $C^1$ (conceivably this could be improved to $C^{\infty}$ for a different choice of approximating sequence).  We obtain a rather clearer global understanding of the examples in Section \ref{accel} than we do of those in Section \ref{bosect}; in fact for a variation on the construction that results in $\psi(\Lambda)$ only being a $C^1$-submanifold rather than a smooth one we are even able to write down an explicit formula for $\psi$ in Example \ref{square}.

Note that for $\psi$ as in either Section \ref{bosect} or Section \ref{accel} 
(corresponding to variations (i) and (iii) of Theorem \ref{exprop}), Theorem 
\ref{main} is applicable to $\psi^{-1}$, and shows that $\psi^{-1}$ cannot 
map a Legendrian submanifold to a non-Legendrian submanifold, whereas by 
Theorem \ref{exprop} $\psi^{-1}$ does map the submanifold $\psi(\Lambda)$ 
that is Legendrian \emph{at a point}\footnote{In fact, inspection of the 
examples shows that $\psi(\Lambda)$ has a codimension-one submanifold 
consisting of points at which it is Legendrian.}  to a non-Legendrian 
submanifold.

\subsection*{Acknowledgements} I am grateful to Will Kazez for helpful conversations, and to Jun Zhang for insightful discussions and useful feedback on the preliminary version of the paper. This work was supported by the NSF through the grant DMS-1509213.

\section{Local rigidity and boundedness}\label{rigsect}

Our proof of Theorem \ref{main} is based on characterizing coisotropic submanifolds of contact manifolds in terms of a notion of local rigidity.  In the symplectic context similar ideas were developed in \cite{U3}; the arguments in the contact context require somewhat more care due to issues relating to conformal factors.  These issues  lead to an addditional hypothesis in our invariance statement, namely Proposition \ref{rigpres}, compared to the symplectic case (\cite[Proposition 1.4]{U3}), and this is the reason for the boundedness hypothesis in Theorem \ref{main}.

Local rigidity is, true to its name, a local property; consequently there is no need to make any compactness or coorientability hypotheses on our contact manifold $(Y,\xi)$, because we can always localize to subsets $U$ having compact closure with $\xi|_{\bar{U}}$ coorientable.

If $W$ is an open subset of a contact manifold $(Y,\xi)$ let $\mathcal{C}_W(Y,\xi)$ denote the space of smooth time-dependent contact vector fields $\mathbb{V}=(V_t)_{t\in [0,1]}$ having compact support contained in $[0,1]\times W$.  For $t\in [0,1]$ we write $\psi^{\mathbb{V},t}$ for the time-$t$ flow of such a vector field.  A choice of contact form $\alpha$ for $\xi|_{W}$ (assuming that one exists, \emph{i.e.} that $\xi|_W$ is coorientable, as will be true for small enough $W$) sets up a one-to-one correspondence between $\mathcal{C}_W(Y,\xi)$ and the space of smooth functions $H\co [0,1]\times Y\to\R$ having compact support contained in $[0,1]\times W$, by setting $H(t,\cdot)=\alpha(V_t)$.

\begin{dfn}\label{disjdef}
Given a contact manifold $(Y,\xi)$, a subset $N\subset Y$,  open subsets $U,W\subset Y$ with $\bar{U}\subset W$ and $N\cap W$ closed as a subset of $W$, and a one-form $\alpha$ on $W$ with $\ker\alpha=\xi|_W$, we define the \textbf{$\alpha$-disjunction energy of $U$ and $N$ rel $W$} as \[ e_{\alpha}^{W}(U,N)=\inf\left\{\left.\int_{0}^{1}\max_W|\alpha(V_t)|dt\right|\mathbb{V}=(V_t)_{t\in[0,1]}\in \mathcal{C}_W(Y,\xi),\,\psi^{\mathbb{V},1}(\bar{U})\cap N=\varnothing\right\}.\]
\end{dfn}

We record some straightforward properties of this quantity, leaving proofs to the reader:

\begin{prop}\label{basice} For $(Y,\xi),N,U,W,\alpha$ as in Definition \ref{disjdef}:
\begin{itemize} \item[(i)] If $\beta=f\alpha$ is another contact form inducing the contact structure $\xi|_W$ on $W$, then \[ \left(\inf_W|f|\right)e_{\alpha}^{W}(U,N)\leq e_{\beta}^{W}(U,N)\leq\left(\sup_W|f|\right)e_{\alpha}^{W}(U,N).\]
\item[(ii)] If $\phi\co Y\to Y'$ is an isocontact embedding between contact manifolds of the same dimension and if $\alpha'$ is a contact form on $\phi(W)$, then \[ e_{\phi^*\alpha'}^{W}(U,N)=e_{\alpha'}^{\phi(W)}(\phi(U),\phi(N)).\]
\item[(iii)] If $N\cap W=N'\cap W$ then $e_{\alpha}^{W}(U,N)=e_{\alpha}^{W}(U,N')$.  
\item[(iv)] If $W\subset W'$, if $\alpha'|_W=\alpha$, and if $N\cap W'$ is closed in $W'$ then $e_{\alpha'}^{W'}(U,N)\leq e_{\alpha}^{W}(U,N)$.
\item[(v)] If $N\subset N'$ with $N'\cap W$ closed in $W$ then $e_{\alpha}^{W}(U,N')\geq e_{\alpha}^{W}(U,N)$.\end{itemize}
\end{prop}

\begin{dfn}\label{rigiddef}
Let $(Y,\xi)$ be a contact manifold, $N\subset Y$ a locally closed subset, and $p\in N$.
We say $p$ is \textbf{locally rigid with respect to $N$} if there is a neighborhood $W$ of $p$ in $Y$ having compact closure such that $N\cap W$ is closed in $W$ and, for every neighborhood $U$ of $p$ with $\bar{U}\subset W$, we have $e_{\alpha}^{W}(U,N)>0$ for one and hence any contact form $\alpha$ for $\xi|_W$ that extends continuously to $\bar{W}$.
\end{dfn}

(That our definition of local rigidity is independent of the choice of contact form on $\bar{W}$ representing $\xi$ is immediate from Proposition \ref{basice} (i) and the requirement in the definition that $\bar{W}$ be compact.)

Here are some quick consequences of Proposition \ref{basice} and Definition \ref{rigiddef}:

\begin{prop}\label{basicr}
For $(Y,\xi)$ a contact manifold, $N,N'\subset Y$ locally closed, and $p\in N$:
\begin{itemize}
\item[(i)] If $N\subset N'$ and $p$ is locally rigid with respect to $N$ then $p$ is locally rigid with respect to $N'$.
\item[(ii)] If $p$ is locally rigid with respect to $N$ and if $(\hat{Y},\hat{\xi})$ is another contact manifold containing a locally closed subset $\hat{N}$ and a point $\hat{p}\in \hat{N}$ such that there is a contactomorphism $\phi$ between neighborhoods $V$ of $p$ in $Y$ and $\hat{V}$ of $\hat{p}$ in $\hat{Y}$ satisfying $\phi(p)=\hat{p}$ and $\phi(N\cap V)=\hat{N}\cap \hat{V}$, then $\hat{p}$ is locally rigid with respect to $\hat{N}$.
\end{itemize}
\end{prop}

\begin{proof}

(i) follows immediately from Proposition \ref{basice} (v).

For (ii), first note that Proposition \ref{basice} (iv) implies that if $p$ is locally rigid with respect to $N$ then $e_{\alpha}^{W}(U,N)>0$ for \emph{all} sufficiently small precompact neighborhoods  $W$ of $p$ and all open $U\subset W$ with $p\in U$ and $\bar{U}\subset W$, and for any contact form $\alpha$ for the restriction of $\xi$ to the closure of such a neighborhood. If necessary, shrink the open subset $V$ in the assumption of (ii) so that $\xi$ has coorientable restriction to a neighborhood of $\bar{V}$.   Choosing a sufficiently small $W$ that in particular is contained in $V$, and letting $\hat{\alpha}$ be an arbitrary contact form for the restriction of $\hat{\xi}$ to a neighborhood of $\phi(\bar{V})$, Proposition \ref{basice} (ii) and (iii) then imply that $e^{\phi(W)}_{\hat{\alpha}}(\phi(U),\hat{N})=e^{W}_{\phi^*\hat{\alpha}}(U,N)$ for every neighborhood $U$ of $p$ having $\bar{U}\subset W$, which suffices to prove the local rigidity of $\hat{p}=\phi(p)$
with respect to $\hat{N}$.
\end{proof}

As mentioned in the introduction, a \textbf{contact homeomorphism} of a contact manifold $(Y,\xi)$ is by definition a homeomorphism $\psi\co Y\to Y$ that is a limit of a sequence of contact diffeomorphisms with respect to the compact-open topology. (Throughout the paper we refer to convergence with respect to the compact-open topology as ``$C^0$-convergence.'')  Note that since the homeomorphism group of $Y$ is a topological group with respect to the compact-open topology by \cite[Theorem 4]{Ar}, the contact homeomorphisms of $(Y,\xi)$ form a subgroup of the homeomorphism group.

\begin{dfn} \label{bddef} Let $(Y,\xi)$ be a contact manifold, let $p\in Y$, and let $\psi\co Y\to Y$ be a contact homeomorphism. 
\begin{enumerate} \item[(A)] We say that $\psi$ is \textbf{bounded below near $p$} if there are: \begin{itemize}\item[(i)] a sequence $\{\psi_m\}$ of contactomorphisms $C^0$-converging to $\psi$; \item[(ii)] a neighborhood $\mathcal{O}$ of $p$ such that the closure $\bar{\mathcal{O}}$ is compact and $\xi|_{\bar{\mathcal{O}}}$ is coorientable;  and \item[(iii)] contact forms $\alpha$ and $\alpha'$ for the restrictions of $\xi$ to neighborhoods of $\bar{\mathcal{O}}$ and $\psi(\bar{\mathcal{O}})$, respectively, such that $(\psi_{m}^{*}\alpha')|_{\bar{\mathcal{O}}}=f_m\alpha|_{\bar{\mathcal{O}}}$ where \begin{equation}\label{belowcrit} \inf_{m\in\Z_+}\inf_{p\in \bar{\mathcal{O}}}|f_m(p)|>0.\end{equation} \end{itemize}
\item[(B)] We say that $\psi$ is \textbf{bounded above near $p$} if there are $\{\psi_m\}$, $\mathcal{O}$, $\alpha,\alpha'$ as in (A)(i-iii) such that, instead of (\ref{belowcrit}), we have $\sup_{m\in \Z_+}\sup_{p\in\bar{\mathcal{O}}}|f_m(p)|<\infty$.
\end{enumerate}
\end{dfn}

Evidently $\psi$ is bounded below near $p$ if and only if $\psi^{-1}$ is bounded above near $\psi(p)$.

\begin{prop}\label{rigpres}
Let $(Y,\xi)$ be a  contact manifold, let $N\subset Y$ be locally closed, and suppose that $p\in N$ is locally rigid with respect to $N$.  If $\psi\co Y\to Y$ is a contact homeomorphism that is bounded below near $p$ then $\psi(p)$ is locally rigid with respect to $\psi(N)$.
\end{prop}

\begin{proof} Let $\{\psi_m\}_{m=1}^{\infty}$ be a sequence of contactomorphims of $Y$, $\mathcal{O}$ a neighborhood of $p$, and $\alpha,\alpha'$ contact forms on $\bar{\mathcal{O}}$ and $\psi(\bar{\mathcal{O}})$ as in Definition \ref{bddef}(A).  
Let $W\subset \mathcal{O}$ be a precompact neighborhood of $p$ such that $N\cap W$ is closed in $W$ and, for every neighborhood $U$ of $p$ with $\bar{U}\subset W$, we have $e_{\alpha}^{W}(U,N)>0$. (As in the proof of Proposition \ref{basicr}(ii) we are free to assume that $W$ is small enough to be contained in $\mathcal{O}$ by Proposition \ref{basice}(iv).) Let $W'=\psi(W)$, and suppose that $U'$ is an arbitrary neighborhood of $\psi(p)$ such that $\overline{U'}\subset W'$. It suffices to show that $e_{\alpha'}^{W'}(U',\psi(N))>0$.  

So suppose that $\mathbb{V}=(V_t)_{t\in [0,1]}$ is a time-dependent contact vector field supported in $W'$ such that $\psi^{\mathbb{V},1}(\overline{U'})\cap \psi(N)=\varnothing$.  Thus \begin{equation}\label{psiconj}(\psi^{-1}\circ\psi^{\mathbb{V},1}\circ\psi)(\overline{\psi^{-1}(U')})\cap N=\varnothing.\end{equation} Using \cite[Theorem 4]{Ar}, the fact that $\psi_m\to \psi$ in the compact-open topology implies that likewise $\psi_{m}^{-1}\circ\psi^{\mathbb{V},1}\circ\psi_m\to \psi^{-1}\circ\psi^{\mathbb{V},1}\circ\psi$ in the compact-open topology.  So for all sufficiently large $m$, (\ref{psiconj}) implies that $(\psi_{m}^{-1}\circ\psi^{\mathbb{V},1}\circ\psi_m)(\overline{\psi^{-1}(U')})\cap N=\varnothing$, \emph{i.e.}, \[ \psi^{\psi_{m*}^{-1}\mathbb{V},1}(\overline{\psi^{-1}(U')})\cap N=\varnothing.\] Moreover since $\mathbb{V}$ has compact support within $[0,1]\times \psi(W)$, if $m$ is sufficiently large (so that $\psi_{m}\circ\psi^{-1}$ is close enough to the identity) then the support of $\psi_{m*}^{-1}\mathbb{V}$ will be contained in $[0,1]\times W$.   Hence for sufficiently large $m$ \begin{equation}\label{pullint} \int_{0}^{1}\max_W|\alpha(\psi_{m*}^{-1}V_t)|dt \geq e_{\alpha}^{W}(\psi^{-1}(U'),N).\end{equation} Now for $x\in Y$, \[ |\alpha'_{\psi_m(x)}(V_t)|=|(\psi_{m}^{*}\alpha')_x(\psi_{m*}^{-1}V_t)|=|f_m(x)\alpha_x(\psi_{m*}^{-1}V_t)|\]  where $f_m\co\bar{\mathcal{O}}\to\R$ are as in Definition \ref{bddef}. So if we write $c=\inf_m\inf_{\bar{\mathcal{O}}}|f_m|$ (which is strictly positive by (\ref{belowcrit})) we find from (\ref{pullint}) that \[ \int_{0}^{1}\max_{W'}|\alpha'(V_t)|dt\geq c\int_{0}^{1}\max_W|\alpha(\psi_{m*}^{-1}V_t)|dt \geq ce_{\alpha}^{W}(\psi^{-1}U',N)>0.\] Since $(V_t)_{t\in [0,1]}$ was arbitrary subject to its support being compactly contained in $[0,1]\times W'$ and its time-one map disjoining $\overline{U'}$ from $\psi(N)$, this suffices to show that $e_{\alpha'}^{W'}(U',\psi(N))>0$.  
\end{proof}


To provide a little more context for Definition \ref{bddef}, we provide a criterion that allows one to see that some contact homeomorphisms $\psi$ are \emph{not} bounded below near a point without checking every sequence of contactomorphisms that $C^0$-converges to $\psi$.  We apply this to some specific examples in Corollary \ref{nonbound}.

\begin{prop}\label{bdcrit} Let $(Y,\xi)$ be a $(2n+1)$-dimensional contact manifold and suppose that a contact homeomorphism $\psi\co Y\to Y$ is bounded below near $p\in Y$.  Then for a sufficiently small neighborhood $U$ of $p$ with $\bar{U}$ compact and for one and hence every choice of contact forms $\alpha$ for $\xi|_{\bar{U}}$ and $\alpha'$ for $\xi|_{\psi(\bar{U})}$ there is $\delta>0$ (depending on $\alpha,\alpha'$) such that, for every nonempty open subset $V\subset U$, we have \begin{equation}\label{vbound} \frac{\int_{\psi(V)}\alpha'\wedge (d\alpha')^{\wedge n}}{\int_V\alpha\wedge (d\alpha)^{\wedge n}}\geq \delta.\end{equation}
\end{prop}

\begin{proof}
We may choose $U$ such that both $U$ and $\psi(U)$ have closures contained in Darboux charts around $p$ and $\psi(p)$ respectively, and take $\alpha$ and $\alpha'$ to be the respective pullbacks of the standard contact form $dz-\sum_jy_jdx_j$ on $\R^{2n+1}$ via these charts.  We also assume that $U$ is contained in a set $\mathcal{O}$ as in Definition \ref{bddef}. Since $B\mapsto \int_B\alpha\wedge (d\alpha)^{\wedge n}$ and $B\mapsto \int_{\psi(B)}\alpha'\wedge (d\alpha')^{\wedge n}$ both define Borel measures on a neighborhood of $\bar{U}$, they are each uniquely determined by their values in the special case where $B$ is the preimage under the Darboux chart of a (sufficiently small) product of intervals.  If $C$ is any such product of intervals, denote by $\frac{1}{2}C$ the product of the intervals with the same centers but half the lengths, and identify $C$ and $\frac{1}{2}C$ with their preimages under our Darboux chart around $p$.  
Since $\alpha$ is identified with the standard contact form on $\mathbb{R}^{2n+1}$ we have \begin{equation}\label{scaleint} \int_{C}\alpha\wedge (d\alpha)^{\wedge n}=2^{2n+1}\int_{\frac{1}{2}C}\alpha\wedge(d\alpha)^{\wedge n}.\end{equation} Let $\psi_m$ be as in Definition \ref{bddef}, with $\psi_{m}^{*}\alpha'=f_m\alpha$ on $\bar{\mathcal{O}}$ where $f_m\co\bar{\mathcal{O}}\to \R$ has $|f_m|\geq c$ for some $c>0$ which is indepedent of $m$.  Since $\psi_m\to\psi$ in the compact-open topology, for any product of intervals $C$ we will have $\psi^{-1}(\psi_m(\frac{1}{2}C))\subset C$ for all $m$ sufficiently large, and hence for large $m$ \begin{align*} \int_{\psi(C)}\alpha'\wedge (d\alpha')^{\wedge n}&\geq \int_{\psi_m(\frac{1}{2}C)}\alpha'\wedge (d\alpha')^{\wedge n}=\left|\int_{\frac{1}{2}C}(f_m\alpha)\wedge (d(f_m\alpha))^{\wedge n}\right|\\& \geq c^{n+1}\int_{\frac{1}{2}C}\alpha\wedge (d\alpha)^{\wedge n}=\frac{c^{n+1}}{2^{2n+1}}\int_{C}\alpha\wedge (d\alpha)^{\wedge n}.\end{align*} So (\ref{vbound}) holds with $\delta=\frac{c^{n+1}}{2^{2n+1}}$ whenever $V$ is any (preimage under our Darboux chart of a) product of open intervals, and hence it also holds with this same value of $\delta$ for arbitrary open $V\subset U$ by standard approximation arguments. 
\end{proof}

It would be interesting to know if the converse to Proposition \ref{bdcrit} also holds.  

\section{Hypertightness and local rigidity for Legendrians}\label{legrigsect}

The key result of this section that is  used in the rest of the paper is Corollary \ref{leglr}, asserting that points on arbitrary Legendrian submanifolds are locally rigid.  This is directly analogous to \cite[Corollary 2.5]{U3} for Lagrangian submanifolds of symplectic manifolds, and the proof strategy is the same: we will prove the result for a restricted class of Legendrians using pseudoholomorphic curve methods (Theorem \ref{hypermain}, analogous to \cite[Lemma 2.4]{U3}), and then exploit the fact that local rigidity is a local property to deduce the result in general via a tubular neighborhood theorem.  To identify the restricted class we introduce the following terminology, borrowed from \cite{CCD}:

\begin{dfn} \label{hyperdef}
A Legendrian submanifold $\Lambda$ of a contact manifold $(Y,\xi)$ is \textbf{hypertight} if there is a contact form $\alpha$ for $\xi$ whose Reeb vector field $R_{\alpha}$ obeys the following properties:
\begin{itemize}\item Every closed orbit of $R_{\alpha}$ is noncontractible.
\item Every Reeb chord for $\Lambda$ (\emph{i.e.}, every $\gamma\co [0,T]\to Y$ such that $\gamma'(t)=R_{\alpha}(\gamma(t))$ and $\gamma(0),\gamma(1)\in \Lambda$) represents a nontrivial element of $\pi_1(Y,\Lambda)$.
\end{itemize}
\end{dfn}

This is a rather restrictive definition, but for our purposes it is sufficient that at least one example with $\Lambda$ and $Y$ both compact exists in every dimension:

\begin{ex}\label{hyperex}
If $Y=ST^*T^{n+1}$ is the unit contangent bundle of $(n+1)$-torus, and if $\Lambda$ is either connected component 
of the unit conormal bundle of the codimension-one torus $\{1\}\times T^n$, then by using the standard contact form whose Reeb flow is the geodesic flow of the flat metric on $T^{n+1}$ we see that $\Lambda$ is a hypertight Legendrian submanifold of $Y$ (and $\dim \Lambda=n$).  (See \cite[Section 3.2]{EHS} for a somewhat more general family of examples.)
\end{ex}

Here is one of our key technical results.
\begin{theorem}\label{hypermain}
If $\Lambda$ is a closed, hypertight Legendrian submanifold of a compact contact manifold $(Y,\xi)$, with contact form $\alpha$ as in Definition \ref{hyperdef},  then for every open subset $U$ of $Y$ with $U\cap \Lambda\neq \varnothing$ we have $e_{\alpha}^{Y}(U,\Lambda)>0$  
\end{theorem}

The proof of Theorem \ref{hypermain} will occupy Section \ref{hyperproof}; let us first extract two consequences from it. Most significanty for the proof of Theorem \ref{main}, we have:

\begin{cor}\label{leglr}
If $(Y,\xi)$ is any contact manifold and $\Lambda\subset Y$ is any Legendrian submanifold then every point on $\Lambda$ is locally rigid with respect to $\Lambda$.
\end{cor}

\begin{proof}  Let $\Lambda'\subset Y'$ be a hypertight Legendrian submanifold of some compact contact manifold $(Y',\xi')$ with $\dim Y'=\dim Y$ (as exists by Example \ref{hyperex}).  

If $p\in \Lambda$ then the Legendrian neighborhood theorem (see \cite[Proposition 43.18]{KM} for a version which does not require compactness of $\Lambda$) gives a contactomorphism $\Psi$ from a neighborhood of $W$ of $p$ in $Y$  to an open set $W'\subset Y'$, such that $\Psi(\Lambda\cap W)=\Lambda'\cap W'$.  It follows immediately from Theorem \ref{hypermain}  that $\Psi(p)$ is locally rigid with respect to $\Lambda'$, and then Proposition \ref{basicr} (ii) applied with $\phi=\Psi^{-1}$ shows that $p$ is locally rigid with respect to $\Lambda$.
\end{proof}

Our other consequence of Theorem \ref{hypermain} concerns the contact version of the Chekanov-Hofer metric on the orbit of a submanifold under the identity component of the contactomorphism group, as considered in \cite{RZ}.  Given a smooth manifold $Y$ with a global contact form $\alpha$ and $\xi=\ker\alpha$, following \cite{She} one defines a norm $\|\cdot\|_{\alpha}$ on the identity component $\mathrm{Cont}_0(Y,\xi)$ of the contactomorphism group by \[ \|\psi\|_{\alpha}=\inf\left\{\left.\int_{0}^{1}\max_Y|\alpha(V_t)|dt\right|\mathbb{V}=(V_t)_{t\in [0,1]}\in\mathcal{C}_Y(Y,\xi),\,\psi^{\mathbb{V},1}=\psi\right\} \] (with notation as in Section \ref{rigsect}).  If $N\subset Y$ is a closed subset, one can then let $\mathcal{L}(N)=\{\psi(N)|\psi\in \mathrm{Cont}_0(Y,\xi)\}$ and, analogously to \cite{Che}, define a pseudometric $\delta_{\alpha}$ on $\mathcal{L}(N)$ by $\delta_{\alpha}(N_1,N_2)=\inf\{\|\psi\|_{\alpha}|\psi(N_1)=N_2\}$.  \cite[Conjecture 1.10]{RZ} states that $\delta_{\alpha}$ is non-degenerate when $N$ is a closed connected Legendrian submanifold.  Theorem \ref{hypermain} quickly implies a special case:

\begin{cor}\label{htr}
Let $\Lambda$ be a closed hypertight Legendrian submanifold of a compact contact manifold $(Y,\xi)$, and let $\xi=\ker\alpha$.  Then the Shelukhin-Chekanov-Hofer pseudometric $\delta_{\alpha}$ is non-degenerate on $\mathcal{L}(\Lambda)$.
\end{cor}

\begin{proof}
By \cite[Proposition 5.1(4)]{RZ} it suffices to check that $\delta_{\alpha}(\Lambda,\Lambda')>0$ whenever $\Lambda'\in\mathcal{L}(\Lambda)$ with $\Lambda'\neq\Lambda$.  Fix such an element $\Lambda'$ and choose $\phi\in\mathrm{Cont}_0(Y,\xi)$ with $\phi(\Lambda)=\Lambda'$; since $\Lambda$ is a closed manifold the fact that $\Lambda'\neq \Lambda$ implies that $\Lambda'\nsubset \Lambda$, so there is an open subset $U$ of $Y$ such that $U\cap\Lambda\neq\varnothing$ and $\phi(\bar{U})\cap \Lambda=\varnothing$.  Let $f\co Y\to (0,\infty)$ be the smooth function such that $\phi^*\alpha=f\alpha$.  We will show that \begin{equation}\label{rzclaim} \delta_{\alpha}(\Lambda,\Lambda')\geq (\min_Yf)e^{Y}_{\alpha}(U,\Lambda),\end{equation} which will suffice to prove the result since Theorem \ref{hypermain} shows that $e^{Y}_{\alpha}(U,\Lambda)>0$.

So let $\psi\in \mathrm{Cont}_{0}(Y,\xi)$ be arbitrary subject to the condition that $\psi(\Lambda)=\Lambda'$.  Since also $\phi(\Lambda)=\Lambda'$ we have $\phi^{-1}\psi(\Lambda)=\Lambda$, and since $\phi(\bar{U})\cap\Lambda=\varnothing$ we obtain \[ \phi^{-1}\psi\phi(\bar{U})\cap \Lambda=\phi^{-1}\psi(\phi(\bar{U})\cap\Lambda)=\varnothing.\]  Thus $\|\phi^{-1}\psi\phi\|_{\alpha}\geq e^{Y}_{\alpha}(U,\Lambda)$.  So we obtain \begin{align*} (\min_Yf)e^{Y}_{\alpha}(U,\Lambda)&\leq (\min_Yf)\|\phi^{-1}\psi\phi\|_{\alpha}=(\min_Yf)\|\psi\|_{\phi^{-1*}\alpha}
\\ &\leq \|\psi\|_{\alpha} \end{align*} where the equality uses \cite[Theorem A(iv)]{She} and the last inequality uses \cite[Lemma 10]{She}.  Since this holds for all $\psi$ with $\psi(\Lambda)=\Lambda'$ we have proven (\ref{rzclaim}).
\end{proof}

\subsection{Proof of Theorem \ref{hypermain}}\label{hyperproof}  What we will in fact show is that there is a positive lower bound on $\int_{0}^{1}|\alpha(V_t)|dt$ for all time-dependent contact vector fields $(V_t)_{t\in [0,1]}$ whose time-one maps disjoin a given compact \emph{pre-Lagrangian} submanifold $L$ from our hypertight Legendrian $\Lambda$ if $L$ and $\Lambda$ have nonempty transverse intersection; this will imply that $e_{\alpha}^{Y}(U,\Lambda)>0$ by choosing $L$ to be contained in $U$.  Recall here that, continuing to write $\dim Y=2n+1$, a pre-Lagrangian submanifold $L$ of $(Y,\xi)$ is an $(n+1)$-dimensional submanifold that is transverse to $\xi$ such that some contact form $\beta$ for $\xi$ obeys $d\beta|_{L}=0$.  Perhaps after reversing the sign of $\beta$ we can write $\beta=e^g\alpha$ for some $g\co Y\to\R$, and then in the symplectization $(\R\times Y,d(e^r\alpha))$ the pre-Lagrangian $L\subset Y$ will lift to a Lagrangian submanifold $\hat{L}=\{(g(q),q)|q\in L\}$.

The key lemma is a lower bound on the Hofer norm of a Hamiltonian diffeomorphism of $\R\times Y$ that is required to disjoin a general compact Lagrangian submanifold $P\subset \R\times Y$ from the Lagrangian submanifold $\R\times \Lambda$. (Eventually we will take $P$ to be a lift $\hat{L}$ of the pre-Lagrangian $L$ mentioned in the previous paragraph.)  If $K\co [0,1]\times (\R\times Y)\to \R$ is a smooth function  we write its Hamiltonian vector field (with respect to the symplectic form $d(e^r\alpha)$) at time $t$ as $Z_{K_t}$, so $d(e^r\alpha)(\cdot,Z_{K_t})=d(K(t,\cdot))$, and we write $\sigma_{K}^{t}$ for the time-$t$ flow of this time-dependent vector field (assuming that this flow exists).

\begin{lemma}\label{rs}  Let $\Lambda$ be a hypertight Legendrian submanifold of a contact manifold $(Y,\xi)$, and let $\alpha\in \Omega^1(Y)$ be as in Definition \ref{hyperdef}.  Also let $P$ be a \textbf{compact} Lagrangian submanifold of the symplectization $(\R\times Y,d(e^r\alpha))$ whose intersection with $\R\times \Lambda$ is nonempty and transverse.  Then there is $\hbar>0$, depending only on $\Lambda,\alpha,P$, such that, for any compactly supported Hamiltonian $K\co [0,1]\times (\R\times Y)\to \R$ with $\sigma_{K}^{1}(P)\cap (\R\times \Lambda)=\varnothing$, we have $\int_{0}^{1}\max_{\R\times Y}|K(t,\cdot)|dt\geq \hbar$.
\end{lemma}

\begin{proof}
Our argument will follow the proof of \cite[Theorem 4.9]{U1} which establishes a similar statement for pairs of compact Lagrangian submanifolds of geometrically bounded symplectic manifolds. See also \cite{Oh} (in the symplectic case) and \cite{Ak} (in the contact case) for earlier related work. As we will see, the assumption that $\Lambda$ is hypertight allows the proof to go through even though symplectizations are not geometrically bounded.

Write $\omega=d(e^r\alpha)$ for the usual symplectic form on the symplectization.  Given a suitably generic family $J=\{J_t\}_{t\in [0,1]}$ of $\omega$-compatible cylindrical almost complex structures on $\R\times Y$, our value $\hbar$ can be taken equal to one-half of the minimum of:\footnote{Because $\omega=d(e^r\alpha)$ is exact there are no nonconstant $J_t$-holomorphic spheres in $\R\times Y$, and because $e^r\alpha$ vanishes on $\R\times \Lambda$ there are no nonconstant $J_1$-holomorphic disks with boundary on $\R\times\Lambda$; thus our list here is shorter than the analogous one in \cite[Proof of Theorem 4.9]{U1}.} 
\begin{itemize} \item the energy of any nonconstant $J_0$-holomorphic disk with boundary on $P$; and   
\item the energy of any nonconstant solution $\tilde{u}\co [0,1]\times \R\to \R\times Y$ to the equation $\frac{\partial \tilde{u}}{\partial s}+J_t\frac{\partial \tilde{u}}{\partial t}=0$ with $\tilde{u}(s,0)\in P$ and $\tilde{u}(s,1)\in \R\times \Lambda$ for all $s\in \R$.
\end{itemize}

The fact that the energies of such objects are bounded away from zero follows from the compactness of $P$, as one can apply a monotonicity lemma such as \cite[Proposition 4.7.2]{Si} in a compact region whose interior contains $P$.

Fix $p\in P\cap (\R\times\Lambda)$.  Let $\mathcal{P}(P,\R\times \Lambda)$ denote the space of continuous paths $\gamma\co [0,1]\to\R\times Y$ with $\gamma(0)\in P$ and $\gamma(1)\in \R\times \Lambda$; this is naturally a topological space, and we treat the constant path $\underline{p}$ at $p$ as the basepoint of this space.  We will be considering maps $\tilde{u}\co \R\times[0,1]\to \R\times Y$ with $\tilde{u}(\R\times\{0\})\subset P,\tilde{u}(\R\times\{1\})\subset \R\times \Lambda$ and such that $\tilde{u}(s,t)\to p$ uniformly in $t$ as $s\to \pm\infty$.  Such a map extends continuously to a map $\bar{u}\co [-\infty,\infty]\times[0,1]\to \R\times Y$ with $\bar{u}(\pm\infty,\cdot)=\underline{p}$, and thus determines a loop $[-\infty,\infty]\to \mathcal{P}(P,\R\times\Lambda)$ based at $\underline{p}$ (sending $s$ to $\bar{u}(s,\cdot)$).  For concision we will say that the original map $\tilde{u}\co \R\times[0,1]\to\R\times Y$ is \emph{contractible} if this associated loop in $\mathcal{P}(P,\R\times\Lambda)$ is trivial in $\pi_1$.
If $\tilde{u} $ is contractible (so in particular $\tilde{u}|_{\R\times\{0\}}$ and $\tilde{u}|_{\R\times\{1\}}$ compactify to contractible loops in the Lagrangian submanifolds $P$ and $\R\times\Lambda$, respectively) then $\int_{\R\times[0,1]}\tilde{u}^*\omega=0$ by Stokes' theorem.

For any $R>0$ let $\beta_R\co \R\to [0,1]$ be a smooth function such that $\beta_R(s)=0$ for $|s|\geq R+1$, $\beta_R(s)=1$ for $|s|\leq R$, and $s\beta'_R(s)\leq 0$ for all $s\in \R$.  Now given $R>0$, $c\in [0,1]$, and a compactly supported smooth $K\co [0,1]\times \R\times Y\to \R$ consider the moduli space $\mathcal{M}_{c,R}^{K}$ of solutions $\tilde{u}\co \R\times[0,1]\to \R\times Y$ to the equation \begin{equation}\label{floereq} \frac{\partial \tilde{u}}{\partial s}+J_t\left(\frac{\partial \tilde{u}}{\partial t}-c\beta_RZ_K\right)=0 \end{equation}  with $\tilde{u}(\R\times\{0\})\subset P$ and $\tilde{u}(\R\times\{1\})\subset \R\times \Lambda$, such that $\tilde{u}(s,t)\to p$ uniformly in $t$ as $s\to \pm\infty$, and such that $\tilde{u}$ is contractible in the sense of the previous paragraph.  Note that such a map $\tilde{u}$ is $J$-holomorphic outside $[-R-1,R+1]\times [0,1]$, and is also $J$-holomorphic outside the preimage of the (compact) union of the supports of the functions  $K(t,\cdot)\co \R\times Y\to \R$.

As in \cite{U1}, standard estimates show that we have, for $\tilde{u}\in \mathcal{M}_{c,R}^{K}$, \begin{align}\label{enest} \nonumber E(\tilde{u} ):=\int_{\R\times [0,1]}\left|\frac{\partial \tilde{u} }{\partial s}\right|_{J_t}^{2}dsdt &\leq \int_{\R\times [0,1]}\tilde{u}^{*}\omega+c\int_{0}^{1}\left(\max K(t,\cdot)-\min K(t,\cdot)\right)dt \\ &\leq 2c\int_{0}^{1}\max|K(t,\cdot)|dt \end{align} where the last inequality follows from $\tilde{u}$ being contractible.    

The key claim is now the following, which depends on the hypertightness assumption on $\Lambda$.

\begin{claim}\label{cpctclaim}
For fixed $K$, there is a compact subset of $\R\times Y$ which contains the image of every $\tilde{u}\in \cup_{c,R}\mathcal{M}_{c,R}^{K}$.
\end{claim}

If one assumes Claim \ref{cpctclaim} the proof of the lemma can be completed just as in \cite[Theorem 4.9]{U1}: for fixed $R$, the space $\mathcal{M}_{0,R}^{K}$ consists of a single transversely-cut-out point, and if $\int_{0}^{1}|K(t,\cdot)|dt<\hbar$ the parametrized moduli space $\cup_{c\in [0,1]}\mathcal{M}_{c,R}^{K}$ will be compact since the energy bound (\ref{enest}) prevents any bubbling or trajectory breaking and Claim \ref{cpctclaim} keeps all solutions in a compact set.  A cobordism argument then shows $\mathcal{M}_{1,R}^{K}\neq \varnothing$ for all $R>0$, and then using Morrey's inequality and the Arzel\`a-Ascoli theorem as in \cite{U1} (with another appeal to Claim \ref{cpctclaim} to ensure boundedness) one obtains from a sequence of elements of $\mathcal{M}_{1,R_m}^{K}$ as $m\to \infty$ a path $\gamma\co[0,1]\to \R\times Y$ with $\gamma(0)\in P,\gamma(1)\in \R\times \Lambda$, and $\gamma'(t)=Z_K(\gamma(t))$.  We will then have an element $\gamma(1)\in \sigma_{K}^{1}(P)\cap (\R\times \Lambda)$ whenever $\int_{0}^{1}\max|K(t,\cdot)|dt<\hbar$, thus proving the lemma modulo Claim \ref{cpctclaim}.

Claim \ref{cpctclaim} follows by arguments like those used in \cite{EHS},\cite{Ak},\cite{RS}; indeed the only difference between our situation and that of \cite[Section 4]{RS} is that we are working with one compact Lagrangian and one cylindrical Lagrangian instead of two cylindrical Lagrangians. Suppose for contradiction that we had $\tilde{u}_m\in \mathcal{M}_{c_m,R_m}^{K}$ (for all $m\in \Z_+$) such that there does not exist any fixed compact set containing the image of every $\tilde{u}_m$.    Let $d>0$ be such that the support of $K$ is contained in $[0,1]\times (-d,d)\times Y$, and also $P\subset (-d,d)\times Y$. For convenience we assume also that $-d$ is a common regular value of all of the $\R$-components of the maps $\tilde{u}_m$ (of course the set of such values is dense by Sard's theorem).  Since $\tilde{u}_m(s,t)\to p\in (-d,d)\times Y$ as $s\to\pm\infty$ and since the Hamiltonian term in (\ref{floereq}) vanishes outside $(-d,d)\times Y$, it follows from the maximum principle (as in \cite[Lemma 4.4]{RS}) that the $\tilde{u}_m$ must all have image contained in $(-\infty,d)\times Y$.  Thus our contradiction assumption implies that, after passing to a subsequence, $\tilde{u}_m$ has image intersecting $(-\infty,a_m)\times Y$ for some sequence $a_m\to-\infty$. 

Write $Z_m=\tilde{u}_{m}^{-1}((-\infty,-d]\times Y)$, and $u_m=\tilde{u}_m|_{Z_m}$.  This is a compact subsurface with boundary and (possibly) corners of $\R\times (0,1]$; one part of the boundary (namely $Z_m\cap (\R\times\{1\})$) maps to $(-\infty,-d]\times \Lambda$ while the other part maps to $\{-d\}\times Y$.  The maps $u_m$ satisfy the hypotheses of \cite[Proposition 4.6]{RS} which gives constraints on the topological behavior of the various subsurfaces $u_{m}^{-1}([R,S]\times Y)$ for $R<S<-d$.  Then, exactly as in \cite[Proposition 4.7]{RS}, a relative version of \cite[Proposition 6.10]{AFM} produces: \begin{itemize} \item a sequence $m_k\to\infty$;
\item subdomains $C_k\subset Z_{m_k}$ which are  all biholomorphic to $[-\ell_k,\ell_k]\times \mathcal{I}$ where $\ell_k\to\infty$ and where $\mathcal{I}$ is either $[0,1]$ or $S^1$, independently of $k$.
\item $\R$-shifts $w_k\co C_k\to \R\times Y$ of the $u_{m_k}|_{C_k}$ (\emph{i.e.}, compositions of $u_{m_k}|_{C_k}$ with the maps $(r,y)\mapsto (r+c_k,y)$ for a suitable sequence $c_k$) such that $\int_{C_k}w_{k}^{*}d\alpha\to 0$, and $\pm r\circ w_k(\pm \ell_k,t)\to\pm\infty$ uniformly in $t$.  In the case that $\mathcal{I}=[0,1]$ we will have $w_k([-\ell_k,\ell_k]\times\{0,1\})\subset \R\times\Lambda$.
\end{itemize}

This then implies, as in \cite[Theorem 5.3]{AFM} and \cite[Proposition 4.8]{RS}, that modulo a sequence of $\R$-shifts, a subsequence of the $u_{m_k}|_{C_k}=\tilde{u}_{m_k}|_{C_k}$ converges in $C^{\infty}_{\mathrm{loc}}$ to a trivial cylinder over a closed Reeb orbit (if $\mathcal{I}=S^1$) or Reeb chord for $\Lambda$ (if $\mathcal{I}=[0,1]$).  In particular this Reeb orbit or chord would represent the same homotopy class as the projection to $Y$ of the image of $\{0\}\times\mathcal{I}$ under $\tilde{u}_{m_k}\circ g_k\co [-\ell_k,\ell_k]\times \mathcal{I}\to M$ for large $k$, where $g_k$ is the biholomorphism that identifies $[-\ell_k,\ell_k]\times \mathcal{I}$ with $C_k\subset Z_{m_k}\subset \R\times(0,1]$. But the projection to $Y$ of the image of any circle or arc under $\tilde{u}_{m_k}\co\R\times(0,1]\to\R\times Y$ of course represents the trivial class in $\pi_1(Y)$ or $\pi_1(Y,\Lambda)$ because $\pi_1(\R\times(0,1])=\pi_1(\R\times(0,1],\R\times\{1\})=\{0\}$.  Thus our Reeb orbit or chord is homotopically trivial, in contradiction with the hypothesis that $\Lambda$ is hypertight.
\end{proof}

\begin{remark}
The assumption that $\Lambda$ is hypertight cannot be completely dispensed with in Lemma \ref{rs}.  If $\dim Y\geq 5$ and $(Y,\ker\alpha)$ is overtwisted, then according to \cite{Mur}\footnote{The equivalence of the condition in \cite{Mur} to overtwistedness is proven in \cite{CMP}. If one allows $Y$ to be noncompact there is a much earlier example in \cite{Mul}.} there exist closed exact Lagrangian submanifolds $P\subset \R\times Y$.  Letting $F_t\co \R\times Y\to\R\times Y$ denote the map $(r,y)\mapsto (r-t,y)$, the exactness of $P$ implies that the Lagrangian submanifolds $F_t(P)$ are all Hamiltonian isotopic.  Arguing as in \cite[Proof of Proposition 11]{Che} this implies that the Chekanov-Hofer pseudometric on the orbit of $P$ is degenerate, and hence identically zero by \cite[Theorem 2]{Che}. So if $\Lambda$ is a Legendrian submanifold (say contained in a small Darboux chart of $Y$) such that $\R\times \Lambda$ intersects $P$ transversely and if $\sigma\co \R\times Y\to \R\times Y$ is a Hamiltonian diffeomorphism such that $\sigma(P)\cap (\R\times \Lambda)=\varnothing$, then there are Hamiltonians $K\co \R\times (\R\times Y)\to \R$ having $\int_{0}^{1}\max|K(t,\cdot)|dt$ as small as one likes such that $\sigma_{K}^{1}(P)=\sigma(P)$ and hence $\sigma_{K}^{1}(P)\cap(\R\times\Lambda)=\varnothing$.

Note that by the main result of \cite{AH} a contact form on (what is now called) a compact overtwisted contact manifold always admits contractible periodic Reeb orbits, and thus cannot contain a hypertight Legendrian.  
\end{remark}

\begin{proof}[Conclusion of the proof of Theorem \ref{hypermain}]
Let $U$ be an open subset of $Y$ with nonempty intersection with $\Lambda$ and let $p\in U\cap \Lambda$.  By \cite[Lemma 4.7]{M19},  there is a compact pre-Lagrangian submanifold $L$ of $Y$ that is contained in $U$ and intersects $\Lambda$ transversely at $p$; by an easy general position argument we can arrange for all other intersections of $L$ and $\Lambda$ to be transverse.  Choose $g\co Y\to\R$ so that $d(e^g\alpha)|_{L}=0$ and $g(p)=0$, and let $\hat{L}=\{(g(q),q)|q\in L\}$, so that $\hat{L}$ is a compact Lagrangian submanifold of $(\R\times Y,d(e^r\alpha))$ whose intersection with $\R\times\Lambda$ is transverse and contains the point $(0,p)$. 

Time-dependent contact vector fields $(V_t)_{t\in[0,1]}$ are in one-to-one correspondence with smooth functions $H\co [0,1]\times Y\to \R$ (by setting $H(t,y)=\alpha_y(V_t)$); given $H\co [0,1]\times Y\to \R$ and $t\in [0,1]$ let $\phi_{H}^{t}$ denote the time-$t$ map of the corresponding time-dependent vector field.

Our goal is then to provide a positive lower bound for $\int_{0}^{1}\max_Y|H(t,\cdot)|dt$ for all $H\co [0,1]\times Y\to \R$ with the property that $\phi_{H}^{1}(\bar{U})\cap \Lambda =\varnothing$.  This property obviously implies that  $\phi_{H}^{1}(L)\cap \Lambda=\varnothing$ where $L$ is the pre-Lagrangian contained in $U$ from the first paragraph of the proof.  A standard calculation (\cite[Section 4]{MSp}) shows that, if $h_t\co Y\to \R$ are the smooth functions obeying $\phi_{H}^{t*}\alpha=e^{h_t}\alpha$, then the (symplectic) Hamiltonian $\hat{H}\co [0,1]\times\R\times Y\to \R$ defined by $\hat{H}(t,r,y)=e^rH(t,y)$ obeys $\sigma_{\hat{H}}^{t}(r,y)=(r-h_t(y),\phi_{H}^{t}(y))$.  Hence in particular $\sigma_{\hat{H}}^{1}(\hat{L})\cap (\R\times \Lambda)=\varnothing$.  This Hamiltonian $\hat{H}$ is not compactly supported; to obtain a compactly supported Hamiltonian one can multiply $\hat{H}$ by a cutoff function $\chi$ that is equal to $1$ on $[0,1]\times [-M,M]\times Y$ for a value $M$ large enough that $|g(y)-h_t(y)|<M$ for all $(t,y)\in [0,1]\times L$, as then $\chi\hat{H}$ and $\hat{H}$ will coincide on a neighborhood of $\cup_{t\in [0,1]}\sigma_{\hat{H}}^{t}(\Lambda)$ and so $\sigma_{\chi\hat{H}}^{1}(\Lambda)=\sigma_{\hat{H}}^{1}(\Lambda)$.

Now as shown in \cite[Proof of Theorem 1.3]{U2}, the function \[ K(t,(r,y))=\chi\hat{H}(1-t,\sigma_{\chi\hat{H}}^{1-t}(\sigma_{\chi \hat{H}}^{1})^{-1}(r,y)), \] which generates the Hamiltonian flow $\sigma_{K}^{t}=\sigma_{\chi \hat{H}}^{1}\circ(\sigma_{\chi\hat{H}}^{1-t})^{-1}$, has the useful properties that $\sigma_{K}^{1}=\sigma_{\chi\hat{H}}^{1}$ and $K(t,\sigma_{K}^{t}(r,y))=\chi\hat{H}(1-t,(r,y))$ for all $t,r,y$.   If we now let $K'\co [0,1]\times (\R\times Y)\to\R$ be a smooth function that is supported on a small neighborhood of $\cup_{t}\{t\}\times \sigma_{K}^{t}(\hat{L})$ and coincides with $K$ on a smaller neighborhood of $\cup_{t}\{t\}\times \sigma_{K}^{t}(\hat{L})$ then we will have $\sigma_{K'}^{1}(\hat{L})=\sigma_{K}^{1}(\hat{L})=\sigma_{\chi\hat{H}}^{1}(\hat{L})$, and (by taking the first neighborhood small enough) we can arrange that, for all $t$, \[ \max_{\R\times Y} |K'(t,\cdot)|\leq \max_{(r,y)\in \hat{L}}|K(t,\sigma_{K}^{t}(r,y))|+\frac{\hbar}{2}=\max_{\hat{L}}|\chi\hat{H}(t,\cdot)|+\frac{\hbar}{2}\] where $\hbar$ is the value from Lemma \ref{rs} (applied with $P=\hat{L}$). But by construction \[ \max_{\hat{L}}|\chi\hat{H}(t,\cdot)|\leq e^{\max_L g}\max_Y|H(t,\cdot)|.\]  

So since $\sigma_{K'}^{1}(\hat{L})\cap (\R\times\Lambda)=\sigma_{\chi\hat{H}}^{1}(\hat{L})\cap (\R\times\Lambda)=\varnothing$, Lemma \ref{rs} gives \[ \hbar\leq \int_{0}^{1}\max_{\R\times Y}|K'(t,\cdot)|dt\leq \frac{\hbar}{2}+e^{\max_L g}\int_{0}^{1}\max_Y|H(t,\cdot)|dt.\]  Since $H$ was arbitrary subject to the assumption that $\phi_{H}^{1}(\bar{U})\cap \Lambda=\varnothing$ this shows that  \[ e_{\alpha}^{Y}(U,\Lambda)\geq \frac{\hbar}{2}e^{-\max_L g}>0.\] 
\end{proof}

\section{Coisotropic submanifolds}\label{coisosect}

The proof of Theorem \ref{main} is completed at the end of this section, after we establish some basic results about coisotropic submanifolds of contact manifolds and their connection to local rigidity.
The literature is somewhat inconsistent as to the definition of a coisotropic submanifold of a contact manifold; our convention in this paper is:

\begin{dfn}\cite{Huang} \label{coisodef} Let $(Y,\xi)$ be a contact manifold, $C\subset Y$ a submanifold, and $p\in C$.  We say that $C$ is \textbf{coisotropic at $p$} if, for one and hence every contact form $\alpha$ for $\xi$ defined on a neighborhood of $p$,  $T_pC\cap\xi_p$ is a coisotropic subspace of the symplectic vector space $(\xi_p,d\alpha_p)$ (\emph{i.e.}, if the $d\alpha_p$-orthogonal complement to $T_pC\cap\xi_p$ is contained in $T_pC\cap \xi_p$).  

We say the submanifold $C\subset Y$ is a coisotropic submanifold if it is coisotropic at $p$ for every $p\in C$.
\end{dfn}

Assuming that $\xi$ is coorientable, \cite[Proposition 3.1]{RZ} shows that $C$ is coisotropic if and only if $\R\times C$ is a coisotropic submanifold of the symplectization of $(Y,\xi)$.  See \cite[Proposition 1.2]{RZ}, as well as Corollaries \ref{coleg} and \ref{corigid} below, for other conditions equivalent to coisotropy.

We quickly observe:
\begin{prop}\label{legco}
Let $C$ be a submanifold of codimension $k$ in a $(2n+1)$-dimensional contact manifold $(Y,\xi)$, and $p\in C$.  If $k>n+1$ then $C$ is not coisotropic at $p$, and if $k=n+1$ then $C$ is coisotropic if and only if $C$ is Legendrian.
\end{prop}

\begin{proof}
For each $p\in C$ the subspace $T_pC\cap \xi_p$ of the $2n$-dimensional vector space $\xi_p$ has codimension $k-1$ if $T_pC\subset \xi_p$, and codimension $k$ otherwise.  Since a coisotropic subspace of $\xi_p$ would have codimension at most $n$ this shows that $C$  can never be coisotropic at $p$  if $k>n+1$, and that if $k=n+1$ then $C$ is coisotropic at $p$ if and only if $T_pC$ is a Lagrangian subspace of $\xi_p$ with respect to the form $(d\alpha)_p$ (where $\alpha$ is a contact form for $\xi$ defined near $p$).   If $C$ is Legendrian (and hence has codimension $n+1$) then $\alpha$ and $d\alpha$ both vanish on $T_pC$ for all $p$ and hence each $T_pC$ is indeed a Lagrangian subspace of $\xi_p=\ker\alpha_p$.  Conversely if the codimension-$(n+1)$ submanifold $C$ is coisotropic then the above discussion shows that $T_pC\subset \xi_p$ for all $p\in C$ and hence that  $C$ is Legendrian.
\end{proof}

In general if $(V,\omega)$ is a symplectic vector space and $W\leq V$ is a subspace we write $W^{\omega}$ for the $\omega$-orthogonal complement: $W^{\omega}=\{v\in V|(\forall w\in W)(\omega(v,w)=0)\}$.  Of course $\dim V=\dim W+\dim W^{\omega}$, and $W$ is coisotropic iff $W^{\omega}\leq  W$.

\begin{lemma}\label{codimwedge}
Let $(V,\omega)$ be a $2n$-dimensional symplectic vector space, and let $W\leq V$ be a subspace of codimension $c\leq n$.  Then $(\omega|_W)^{\wedge (n-c)}\neq 0$, and $(\omega|_W)^{\wedge(n-c+1)}=0$ if and only if $W$ is a coisotropic subspace.
\end{lemma}

\begin{proof}
Choose any subspace $X\leq W$ such that $W=(W\cap W^{\omega})\oplus X$.  It is then straightforward to see that $\omega$ restricts nondegenerately to $X$, and that if $\pi\co W\to X$ is the projection with kernel $W\cap W^{\omega}$ then $\omega|_W=\pi^{*}(\omega|_X)$.  So $X$ has some even dimension $2j$, in which case $(\omega|_X)^{\wedge j}\neq 0$ while $(\omega|_X)^{\wedge (j+1)}=0$, and hence $(\omega|_W)^{\wedge j}\neq 0$ while $(\omega|_W)^{\wedge (j+1)}=0$. 

Now \[ 2j=\dim W-\dim(W\cap W^{\omega})\geq \dim W-\dim W^{\omega}=2n-2c,\] with equality holding iff $W\cap W^{\omega}=W^{\omega}$, \emph{i.e.} iff $W$ is coisotropic.  Since in any event $j\geq n-c$ and, as already noted, $(\omega|_W)^{\wedge j}\neq 0$, this shows that we have $(\omega|_{W})^{\wedge(n-c)}\neq 0$ for arbitrary $W$.  If $W$ is not coisotropic then $j\geq n-c+1$ and so likewise $(\omega|_W)^{\wedge(n-c+1)}\neq 0$, while if $W$ is coisotropic then $n-c+1=j+1$ and hence $(\omega|_W)^{\wedge(n-c+1)}=(\omega|_W)^{\wedge(j+1)}=0$.
\end{proof}

\begin{prop}\label{dlambda}
Let $C$ be a submanifold of codimension $k\leq n$ in a $(2n+1)$-dimensional contact manifold $(Y,\xi)$, let $p\in C$, let $U$ be a neighborhood of $p$ and $\alpha\in \Omega^1(U)$ a contact form for $\xi|_U$, and write $\lambda=\alpha|_{C\cap U}$.  \begin{itemize}
\item If $\lambda_p=0$, then $(d\lambda)_{p}^{\wedge(n-k+1)}\neq 0$, and $(d\lambda)_{p}^{\wedge(n-k+2)}=0$ if and only if $C$ is coisotropic at $p$.
 \item If $\lambda_p\neq 0$, then $\lambda_p\wedge(d\lambda)_{p}^{\wedge(n-k)}\neq 0$, and $\lambda_p\wedge(d\lambda)_{p}^{\wedge(n-k+1)}=0$ if and only if $C$ is coisotropic at $p$.
\end{itemize}
\end{prop}

\begin{proof}
If $\lambda_p=0$, then $T_pC=T_pC\cap \xi_p$ is a codimension-$(k-1)$ subspace of $\xi_p$, so the statement follows from Lemma \ref{codimwedge}.

If instead $\lambda_p\neq 0$, then since $\dim(T_pC)-\dim(T_pC\cap\xi_p)=\dim(T_pY)-\dim(\xi_p)=1$ we see that $T_pC\cap \xi_p$ has codimension $k$ in $\xi_p$.  So applying Lemma \ref{codimwedge} shows that $\left((d\lambda)|_{T_pC\cap \xi_p}\right)^{\wedge(n-k)}\neq 0$, and that $\left((d\lambda)|_{T_pC\cap \xi_p}\right)^{\wedge(n-k+1)}=0$ if and only if $C$ is coisotropic at $p$.  If we fix an an arbitrary element $v$ of $T_pC\setminus \xi_p$ then, for $j\in \mathbb{N}$, the $(2j+1)$-form $\lambda_p\wedge(d\lambda)_p$ on $T_pC$ is zero iff it evaluates to $0$ on all tuples of form $(v,w_1,\ldots,w_{2j})$ where $w_1,\ldots, w_{2j}\in T_pC\cap \xi_p$.  So what we have shown about powers of $(d\lambda)|_{T_pC\cap \xi_p}$ implies that indeed $\lambda_p\wedge(d\lambda)_{p}^{\wedge(n-k)}\neq 0$ and that $\lambda_p\wedge(d\lambda)_{p}^{\wedge(n-k+1)}=0$ iff $C$ is coisotropic at $p$.
\end{proof}

\begin{prop}\label{intleg}
Let $C$ be a submanifold of a contact manifold $(Y,\xi)$ and $p\in C$, and suppose that there is a Legendrian submanifold $\Lambda$ of $Y$ such that $p\in \Lambda\subset C$.  Then $C$ is coisotropic at $p$.
\end{prop}

\begin{proof}
Under the assumption we have $T_p\Lambda=T_p\Lambda\cap\xi_p\subset T_pC\cap \xi_p$ with $T_p\Lambda$ a Lagrangian subspace of $\xi_p$ and hence, taking $(d\alpha_p)$-orthogonal complements within $\xi_p$ where $\alpha$ is a contact form defined near $p$, \[ (T_pC\cap \xi_p)^{d\alpha_p}\subset (T_p\Lambda)^{d\alpha_p}=T_p\Lambda\subset T_pC\cap \xi_p.\]
\end{proof}

Below in Proposition \ref{legexists}  we will establish a partial converse to Proposition \ref{intleg}; we begin with observations concerning   flows of certain contact vector fields.  For this purpose it is convenient to identify contact vector fields with Hamiltonians, which requires choosing a contact form, so the next couple of lemmas will require the ambient contact manifold to be coorientable; while we ultimately want to prove certain statements that do not require a coorientability hypothesis, these statements are local so this does not pose a serious problem.

    Recall that if $\alpha$ is a contact form on a smooth manifold $Y$ and $\xi=\ker\alpha$ the Hamiltonian vector field of a smooth function $H\co Y\to \R$ is the vector field $X_H$ characterized uniquely by the properties that $\alpha(X_H)=H$ and $\iota_{X_H}d\alpha=dH(R_{\alpha})\alpha-dH$ where $R_{\alpha}$ is the Reeb field of $\alpha$.  

\begin{lemma}\label{coisovanish}
Let $C$ be a submanifold of $Y$, let $\alpha$ be a contact form on $Y$ with $\xi=\ker\alpha$, and  let $H\co Y\to \R$ be smooth.  Then $(X_H)_q\in (T_qC\cap \xi_q)^{d\alpha|_{\xi_q}}$ for all $q\in C$ if and only if $H|_C=0$.
\end{lemma}

\begin{proof}
The forward implication is trivial: if $(X_H)_q\in (T_qC\cap \xi_q)^{d\alpha|_{\xi_q}}$ for all $q\in C$ then in particular $(X_H)_q\in \xi_q=\ker\alpha_q$ and so, for all $q\in C$, $H(q)=\alpha_q(X_H)=0$.

Conversely if $H|_C=0$ then for each $q\in C$ we have $\alpha_q(X_H)=0$ and so $(X_H)_q\in \xi_q$, and moreover, for each $v\in T_qC\cap \xi_q$, \[ d\alpha(X_H,v)=dH(R_{\alpha})\alpha(v)-dH(v)=0 \] where the first term vanishes because $v\in \xi_q$ and the second vanishes because $v\in T_qC$. 
\end{proof}

\begin{lemma}\label{mult} If $C$ is a submanifold of a smooth manifold $Y$ equipped with a contact form $\alpha$ and if $H\co Y\to \R$ is smooth with $H|_C=0$, then for any other smooth function $f\co Y\to \R$ we have $(X_{fH})_q=f(q)(X_H)_q$ for all $q\in C$.
\end{lemma}

\begin{proof} For any $q\in Y$, the tangent vector $(X_{fH})_q\in T_qY$ is uniquely characterized by the properties that $\alpha((X_{fH})_q)=f(q)H(q)$ and $\iota_{(X_{fH})_q}d\alpha|_{\xi_q}=-d(fH)|_{\xi_q}$ where $\xi_q=\ker\alpha_q$ so we just need to check that $f(q)(X_H)_q$ obeys the same properties when $q\in C$.  This is clear since, due to the assumption that $H|_C=0$, we have $d(fH)_q=f(q)dH_q$.
\end{proof}

\begin{prop}\label{legexists} Let $C$ be a coisotropic submanifold of a contact manifold $(Y,\xi)$ and let $p\in C$ with $T_pC\not\subset \xi_p$.  Then there is a Legendrian submanifold $\Lambda$ of $Y$ such that $p\in \Lambda\subset C$.
\end{prop}

\begin{proof}
A concise summary of the proof is that, for a suitably small neighborhood $W$ of $p$ with $\alpha$ a contact form for $\xi|_W$, the neighborhood $C\cap W$ of $p$ in $C$ can be ``coisotropically reduced,'' yielding a projection $\pi\co C\cap W\to Z$ where $Z$ comes equipped with a contact form $\beta$ having $\pi^*\beta=\alpha|_{C\cap W}$, and then we can take $\Lambda=\pi^{-1}(\Lambda_0)$ for a Legendrian submanifold $\Lambda_0\subset Z$ that passes through $\pi(p)$.  (Below $Z$ will be constructed as a local transversal to the foliation spanned by $(T(C\cap W)\cap \xi)^{d\alpha|_{\xi}}$ and $\beta$ will just be $\alpha|_Z$.  See \cite[Theorem 5.3.30]{AM} for an analogous construction in the symplectic case, and for the contact case compare \cite[Theorem 13]{LdL}, though note that the definition of coisotropy therein is slightly different from ours.)

We now give full details.  Choose a neighborhood $W$ of $p$ and smooth functions $H_1,\ldots H_k\co W\to \R$ such that $C\cap W$ is given as a regular level set $C\cap W=\{H_1=\cdots=H_k=0\}$. (In particular we are assuming the $dH_j$ to be pointwise-linearly-independent along $C$, so $\dim C=\dim Y-k$.)  Shrinking $W$ if necessary, let $\alpha\in \Omega^1(W)$ have $\ker\alpha=\xi|_W$, and assume that $T_qC\not\subset \xi_q$ for all $q\in C\cap W$.  Note that this implies that the restrictions $dH_j|_{\xi_q}$ are linearly independent at each $q\in C\cap W$: choosing $v\in T_qC\setminus \xi_q$, a linear combination $H=\sum_jc_jH_j$ automatically has $dH(v)=0$, so if $(dH)_q|_{\xi_q}=0$ then $(dH)_q$ vanishes identically on $T_qY$ and hence the coefficients $c_j$ are all zero.  

By Lemma \ref{coisovanish}, we have $(X_{H_j})_q\in (T_qC\cap \xi_q)^{d\alpha|_{\xi_q}}$ for all $q\in C\cap W$ and each $j=1,\ldots,k$.  Because each of the $T_qC\cap \xi_q$ (for $q\in C\cap W$) has codimension $k$ in $\xi_q$, each of the $(T_qC\cap \xi_q)^{d\alpha|_{\xi_q}}$  is a $k$-dimensional subspace of $T_qC$; thus we have a rank-$k$ distribution $(TC\cap \xi)^{d\alpha|_{\xi}}$ on $C\cap W$, of which each $X_{H_j}|_{C\cap W}$ is a section.  These sections $X_{H_1}|_{C\cap W},\ldots,X_{H_k}|_{C\cap W}$ are moreover linearly independent, since $\iota_{X_{H_j}}d\alpha|_{\xi}=-dH_j|_{\xi}$ and as noted at the end of the previous paragraph the $dH_j|_{\xi}$ are linearly independent along $C\cap W$.  Thus our distribution $\mathcal{F}:=(TC\cap \xi)^{d\alpha|_{\xi}}$ on $C\cap W$ is the pointwise-linearly-independent span of the restrictions of the vector fields $X_{H_1},\ldots,X_{H_k}$ to $C\cap W$.  

We next claim that this distribution $\mathcal{F}$ is involutive.  Indeed letting $\{\cdot,\cdot\}$ denote the contact Poisson bracket as in \cite[Remark 3.5.18]{MS3}, one has $[X_{H_i},X_{H_j}]=X_{\{H_i,H_j\}}$ for any $i,j\in\{1,\ldots,k\}$, and then by \cite[Proposition 1.2]{RZ} $\{H_i,H_j\}|_{C\cap W}=0$, whence $X_{\{H_i,H_j\}}$ is a section  of $(TC\cap \xi)^{d\alpha|_{\xi}}$ by Lemma \ref{coisovanish}.

By the Frobenius theorem, the involutivity of $\mathcal{F}$ implies that, after perhaps shrinking $W$, we can find vector fields $V_i=\sum_j f_{ij}X_{H_j}$ (for $i=1,\ldots,k$ and some smooth functions $f_{ij}\co W\to \R$) which continue to span $\mathcal{F}$ pointwise in $C\cap W$ and which obey $[V_i,V_j]=0$. (Specifically the $V_i$ may be identified with coordinate vector fields for a flat chart for $(TC\cap \xi)^{d\alpha|_{\xi}}$ around $p$.)  By Lemma \ref{mult},  we have $V_i|_C=\sum_jX_{f_{ij}H_j}|_C$, so if $K_i=\sum_jf_{ij}H_j$ the functions $K_1,\ldots,K_k$ vanish along $C\cap W$ and have the property that, along $C\cap W$, the $X_{K_j}$ pairwise commute and span $\mathcal{F}=(TC\cap\xi)^{\xi|_{d\alpha}}$.

Now, possibly after shrinking $W$ again, let $Z$ be a codimension-$k$ submanifold of $C\cap W$ that passes through our point $p$ and is transverse (in $C\cap W$)  to the $k$-dimensional foliation spanned by $\mathcal{F}$. In particular for each $q\in Z$, $T_qZ\not\subset \xi_q$.   Now by Proposition \ref{dlambda}, the $(2n+1-2k)$-form $\alpha\wedge(d\alpha)^{\wedge(n-k)}$ has nowhere-vanishing restriction to $C\cap W$.  So for $q\in Z$ we can find $v\in T_qZ\setminus \xi_q$ and $w_1,\ldots,w_{2(n-k)}\in T_qC\cap \xi_q$ such that \[ 0\neq \alpha\wedge(d\alpha)^{\wedge(n-k)}(v,w_1,\ldots,w_{2n-2k})=\alpha(v)(d\alpha)^{\wedge(n-k)}(w_1,\ldots,w_{2n-2k}).\]  But since $\mathcal{F}$ is contained in and $d\alpha$-orthogonal to $TC\cap\xi$, and since $T_qC\cap \xi_q=(TZ\cap\xi_q)\oplus \mathcal{F}_q$, replacing the $w_i$ above by their projections to $T_qZ\cap \xi_q$
will not change the property that $\alpha(v)(d\alpha)^{(n-k)}(w_1,\ldots,w_{2n-2k})\neq 0$.  This proves that $\left(\alpha\wedge(d\alpha)^{(n-k)}\right)|_Z$ is a nowhere-vanishing $(2n+1-2k)$-form on $Z$.  But $\dim Z=\dim C-k=2n+1-2k$, so what we have just shown is that $\alpha|_Z$ is a contact form.

Now (for instance by the contact Darboux theorem), within the contact manifold $(Z,\ker\alpha|_Z)$ we can take a Legendrian submanifold $\Lambda_Z$ of $Z$ that passes through the point $p$.  (Thus $\dim \Lambda_Z=n-k$.)  In general letting $\phi_{H}^{t}$ denote the Hamiltonian flow of the contact Hamiltonian $H$ with respect to the contact form $\alpha$ on $W$ we now take \[ \Lambda=\left\{\phi_{K_1}^{t_1}\circ\cdots\circ\phi_{K_k}^{t_k}(x)|x\in \Lambda_Z\cap W',(t_1,\ldots,t_k)\in U\right\}\] for neighborhoods $W'$ of $p$ in $C$ and $U$ of the origin in $\R^k$ that are sufficiently small for the relevant Hamiltonian flows to be defined and for $\Lambda$ as given above to be an embedded submanifold.  The tangent space to $\Lambda$  at $\phi_{K_1}^{t_1}\circ\cdots\circ\phi_{K_k}^{t_k}(x)$ is spanned by the vector fields $X_{K_k}$ (which lie in $\ker\alpha$) together with the image under the linearization of $\phi_{K_1}^{t_1}\circ\cdots\circ\phi_{K_k}^{t_k}$ of the tangent space $T_x\Lambda_Z$, and this image is annihilated by $\alpha$ because $T_x\Lambda_Z\subset \xi_x$ while the $\phi_{K_k}^{t_k}$ are contactomorphisms.  So $\Lambda$ is an $n$-dimensional submanifold of $C$ containing $p$ with $\alpha|_{T\Lambda}=0$, as desired.
\end{proof}

\begin{remark}  
The assumption that $T_pC\not\subset \xi_p$ in Lemma \ref{legexists} cannot be completely discarded, as can already be seen in the case that $\dim Y=3$ and $\dim C=2$.  In this case the Legendrian submanifolds of  $Y$ that are contained in $C$ coincide away from the singular set $\{p\in C|T_pC\subset \xi_p\}$ with the leaves of the characteristic foliation (\emph{i.e.} the foliation tangent to $TC\cap \xi$).  If $p$ is an isolated point of this singular set then it may not be possible to find a one-dimensional smooth submanifold passing through $p$ that coincides away from the singular set with a union of such leaves---for example if the foliation has a spiral source at $p$ then any smoothly embedded arc through $p$ will have infinitely many transverse intersections with each leaf that approaches $p$.
\end{remark}

\begin{cor}\label{coleg} A submanifold $C$ of a contact manifold $(Y,\xi)$ is coisotropic if and only if there is a dense, relatively open subset $U\subset C$ such that for each $p\in U$ there exists a Legendrian submanifold $\Lambda$ of $Y$ such that $p\in \Lambda \subset C$.
\end{cor}

\begin{proof}   As before write $\dim Y=2n+1$ and $k=\dim Y-\dim C$. If $k>n+1$ the statement of the corollary is vacuous since a nonempty submanifold of codimension greater than $n+1$ can neither be coisotropic nor contain a nonempty Legendrian submanifold.  If $k=n+1$ and $C$ is coisotropic then $C$ is Legendrian by Proposition \ref{legco} so we can take $U=\Lambda=C$. Conversely if $k=n+1$ and $p\in \Lambda\subset C$ with $\Lambda$ Legendrian then by dimensional considerations $\Lambda$ contains an open-in-$C$ neighborhood of $p$, so that $T_pC\subset \xi_p$. So if the set of points $p$ admitting such a Legendrian is dense in $C$ then for any open set $V$ on which $\xi$ can be written as $\ker\alpha$ it holds that $\alpha|_{T(C\cap V)}$ vanishes on a dense subset of $C\cap V$ and hence on all of $C\cap V$, whence $C$ is Legendrian and thus coisotropic.  So assume for the rest of the proof that $k\leq n$.

In this case, we claim that the set of points $p\in C$ such that $T_pC\subset \xi_p$ has empty interior.  If this were false there would be a nonempty open subset $V\subset Y$ intersecting $C$ on which $\xi|_V=\ker\alpha$ for some $\alpha\in \Omega^1(V)$ such that $\lambda:=\alpha|_{C\cap V}$ vanished throughout $C\cap V$, in which case $d\lambda$ would also vanish  throughout $C\cap V$. But by Proposition \ref{dlambda} we have $(d\lambda)^{\wedge(n-k+1)}\neq 0$ and so (since $k\leq n$) $d\lambda\neq 0$.  So indeed $U=\{p\in C|T_pC\nsubset \xi_p\}$ is open and dense in $C$ (regardless of whether $C$ is coisotropic), and by Proposition \ref{legexists} if $C$ is coisotropic then for each $p\in U$ there is a Legendrian $\Lambda$ with $p\in \Lambda\subset C$.  

Conversely, if $W\subset C$ is an open and dense subset such that each $p\in W$ admits a Legendrian $\Lambda$ with $p\in \Lambda\subset C$, then $C$ is coisotropic at $p$ for each $p\in W$ by Proposition \ref{intleg}.  So letting $U=\{p\in C|T_pC\nsubset \xi_p\}$ as above and considering any sufficiently small open $V$ and $\alpha\in \Omega^1(V)$ with $\xi|_V=\ker\alpha$, for each $p\in U\cap V\cap W\subset C$,  Proposition \ref{dlambda} shows that $\lambda_p\wedge(d\lambda)^{\wedge(n-k+1)}_{p}=0$ where $\lambda=\alpha|_{C\cap V}$.  But since $U\cap V\cap W$ is dense in $C\cap V$ (being the intersection of two open dense sets $U\cap V$ and $W\cap V$) this implies that $\lambda\wedge (d\lambda)^{\wedge(n-k+1)}=0$ everywhere on $C\cap V$, and hence also that $(d\lambda)^{\wedge(n-k+2)}=0$ everywhere on $C\cap V$.  Another appeal to Proposition \ref{dlambda} thus shows that $C\cap V$ is coisotropic at $p$ for every $p\in C\cap V$.  Allowing $V$ to vary through open subsets on which $\xi|_V$ is coorientable thus shows that $C$ is coisotropic.
\end{proof}

\begin{prop}\label{nonco}
Let $C$ be a submanifold of a contact manifold $(Y,\xi)$  and suppose that $p\in C$ is locally rigid with respect to $C$.  Then $C$ is coisotropic at $p$.
\end{prop}

\begin{proof}
Let $W$ be a neighborhood of $p$ that is sufficiently small for $C\cap W$ to be closed as a subset of $W$ and for there to be a contact form $\alpha$  for $\xi|_W$.  Suppose that $C$ is not coisotropic at $p$, so that there is $v\in \xi_p$ such that $v\in (T_pC\cap \xi_p)^{d\alpha|_{\xi_p}}\subset \xi_p$ while $v\notin T_pC$.  We will find a neighborhood $U$ of $p$ with $\bar{U}\subset W$ such that $e_{\alpha}^{W}(U,C)=0$; in view of Proposition \ref{basice}(iv) and the fact that $W$ is arbitrary subject to being sufficiently small, this will prove that $p$ is not locally rigid with respect to $C$.

To do this, following the strategy of \cite[Lemma 4.3]{U1} and \cite[Proposition 7.3]{RZ}, let $H\co Y\to \R$ be a smooth function having compact support contained in $W$ such that $H|_C=0$ and $dH_p(v)>0$, as is possible since $v$ is not tangent to $C$.  The contact Hamiltonian vector field $X_H$ of $H$ on $W$ with respect to $\alpha$ will then obey $\alpha(X_H)=0$ at all points of $C\cap W$ and, using that $v\in \xi_p$, $d\alpha(X_H,v)=-dH_p(v)\neq 0$. Thus $(X_H)_p\in \xi_p$ but $(X_H)_p\notin T_pC\cap \xi_p$, since $d\alpha(\cdot,v)$ restricts to zero on $T_pC\cap\xi_p$.  Thus for sufficiently small positive $t$ we will have $\phi_{H}^{t}(p)\notin C$; replacing $H$ by $H/t$ if necessary we may as well assume that $\phi_{H}^{1}(p)\notin C$.  Since $C\cap W$ is closed as a subset of $W$ this implies that there is an open set $U$ around $p$, which we can assume to obey $\bar{U}\subset W$, such that $\phi_{H}^{1}(\overline{U})\cap C=\varnothing$.  The proof will be complete when we show that $e_{\alpha}^{W}(U,C)=0$.

Choose a sequence of smooth functions $\beta_k\co \R\to\R$ such that: \begin{itemize}
\item $\beta_k(s)=s$ whenever $|s|\geq 2/k$;
\item $\beta_k(s)=0$ whenever $|s|\leq 1/k$; and 
\item $0\leq \beta'_k(s)\leq 3$ for all $s$.
\end{itemize}

The functions $\beta_k\circ H$ are supported in $W$ and each vanish throughout a ($k$-dependent) neighborhood of $C$, and so $\phi_{\beta_k\circ H}^{t}$ will restrict to the identity on $C$ for each $t$ and $k$.  So the fact that $\phi_{H}^{1}(\overline{U})\cap C=\varnothing$ implies that \[ (\phi_{\beta_k\circ H}^{1})^{-1}\circ\phi_{H}^{1}(\overline{U})\cap C=(\phi_{\beta_k\circ H}^{1})^{-1}(\phi_{H}^{1}(\overline{U})\cap C)=\varnothing.\]  If we write $f_{k,t}\co W\to \R$ for the smooth functions such that $\phi_{\beta_{k}\circ H}^{t*}\alpha=f_{k,t}\alpha$, then by a standard calculation as in \cite[Lemma 2.2]{MSp} the isotopy $\{(\phi_{\beta_k\circ H}^{t})^{-1}\circ\phi_{H}^{t}\}_{t\in [0,1]}$ is generated by the unique contact vector field $(V_{k,t})_{t\in [0,1]}$ that obeys \begin{equation}\label{betacompV} \alpha(V_{k,t})=\frac{1}{f_{k,t}}\left(H-\beta_k\circ H\right)\circ \phi_{\beta_k\circ H}^{t}.
\end{equation}

This is slightly more complicated than the situation in \cite[Proof of Proposition 7.3]{RZ} because the order in which we need to compose our diffeomorphisms is opposite to theirs, leading to a factor $\frac{1}{f_{k,t}}$ in (\ref{betacompV}) that depends on $k$, but these factors can be estimated as follows.  The Lie derivative of $\alpha$ along the Hamiltonian vector field  $X_{\beta_k\circ H}$ is given by \[ L_{X_{\beta_k\circ H}}\alpha=d(\beta_k\circ H)+(\iota_{R_{\alpha}}d(\beta_k\circ H))\alpha-d(\beta_k\circ H)=\left((\beta'_k\circ H)\iota_{R_{\alpha}}dH\right)\alpha \] where $R_{\alpha}$ is the Reeb vector field of $\alpha$, and thus we have \[ \log f_{k,t}=\int_{0}^{t}\left((\beta'_k\circ H)\iota_{R_{\alpha}}dH\right)\circ \phi_{\beta_k\circ H}^{s}ds.\]  So choosing $M>0$ such that $|\iota_{R_{\alpha}}dH|\leq M$ everywhere on $W$, our assumption that $0\leq \beta'_k\leq 3$ shows that we have \[ |\log f_{k,t}(x)|\leq 3M\mbox{ for all }k\in \Z_+,t\in [0,1],x\in Y.\] Moreover our construction of $\beta_k$ also ensures that $|H-\beta_k\circ H|\leq \frac{2}{k}$ everywhere.  Hence (\ref{betacompV}) yields \[ |\alpha(V_{k,t})|\leq \frac{2e^{3M}}{k} \] everywhere, where the constant $M$ depends on $H$ and $\alpha$ but not on $k$.  Since the time-one flow $(\phi_{\beta_k\circ H}^{1})^{-1}\circ \phi_{H}^{1}$ disjoins $\bar{U}$ from $N$ this proves that $e_{\alpha}^{W}(U,C)\leq \frac{2e^{3M}}{k}$ for all positive integers $k$, and hence that $e_{\alpha}^{W}(U,C)=0$, as desired.
\end{proof}

\begin{cor}\label{corigid}
Let $C$ be a submanifold of a contact manifold $(Y,\xi)$.  Then $C$ is coisotropic if and only if there is a relatively open and dense subset $U\subset C$ such that every point $p\in U$ is locally rigid with respect to $C$.
\end{cor}

\begin{proof}
If $C$ is coisotropic then all of the points in the relatively open and dense subset from Corollary \ref{coleg} will be locally rigid by Corollary \ref{leglr} and Proposition  \ref{basicr}(i).  On the other hand if $C$ is not coisotropic we claim that the set of points at which it fails to be coisotropic contains a nonempty open set. Let $V$ be an open subset of $Y$ such that $\xi|_V=\ker\alpha$ with $\alpha\in \Omega^1(V)$ and such that $C\cap V$ contains a point at which $C$ is not coisotropic.  Write $\lambda=\alpha|_{C\cap V}$ and $k=\dim Y-\dim C$. If we had $\lambda_p\wedge (d\lambda)_p^{\wedge(n-k+1)}= 0$ at every point of $C\cap V$ then taking a derivative would show $(d\lambda)^{\wedge(n-k+2)}=0$ throughout $C\cap V$ which is impossible by Proposition \ref{dlambda} and our assumption on $V$. So there must be some point $p\in C\cap V$ at which $\lambda_p\wedge (d\lambda)_p^{\wedge(n-k+1)}\neq  0$.  But then $\lambda\wedge (d\lambda)^{(n-k+1)}$ (and hence also $\lambda$) is nowhere vanishing on a neighborhood $W$ of $p$ in $C$, and so for each $q\in W$, $C$ is not coisotropic at $q$ by another application of Proposition \ref{dlambda}.  By Proposition \ref{nonco} this implies that, for each $q$ in the nonempty relatively open set $W$, $q$ is not locally rigid with respect to $C$; thus $C$ cannot contain a dense set of points each of which is locally rigid.
\end{proof}


\begin{proof}[Proof of Theorem \ref{main}]
Let $\psi\co Y\to Y$ be a contact homeomorphism and $C\subset Y$ a coisotropic submanifold such that $\psi(C)$ is a smooth submanifold and $\psi$ is bounded below near every point of $C$.  By Corollary \ref{corigid}, there is a dense and relatively open subset $U\subset C$ such that each point of $C$ is locally rigid with respect to $C$.  Then Proposition \ref{rigpres} shows that each point of $\psi(U)\subset \psi(C)$ (which is open and dense since $\psi|_C$ is a homeomorphism) is likewise locally rigid with respect to $\psi(C)$.  But then $\psi(C)$ is coisotropic by Corollary \ref{corigid}.
\end{proof}

\section{Instability of coisotropy at a  point} \label{exsect}

This section contains the examples which prove Theorem \ref{exprop}, showing that a contact homeomorphism $\psi$ can map a submanifold that is not coisotropic at some point $p$ to one which is coisotropic at $\psi(p)$.  Our constructions are local in nature, taking place in an open subset of $\R^{2n+1}$ in Section \ref{bosect} and in an open subset of the one-jet bundle of the $n$-torus in Section \ref{accel}; we always use the contact form \[ \alpha=dz-\sum_{j=1}^{n}y_jdx_j\] in either case (with $x_j$ valued in $\R$ in Section \ref{bosect} and in $\R/\Z$ in Section \ref{accel}).  The Hamiltonian vector field $X_H$ of a smooth function $H$ with respect to this contact form $\alpha$ is then given by \begin{equation}\label{hamform} X_H=-\sum_j\frac{\partial H}{\partial y_j}\partial_{x_j}+\sum_j\left(\frac{\partial H}{\partial x_j}+y_j\frac{\partial H}{\partial z}\right)\partial_{y_j}+\left(H-\sum_jy_j\frac{\partial H}{\partial y_j}\right)\partial_z.\end{equation}  One then has $L_{X_H}\alpha=\frac{\partial H}{\partial z}\alpha$, and so if $\phi_{H}^{t}$ is the time-$t$ map of the Hamiltonian flow of $H$ then the function $f$ obeying $\phi_{H}^{1*}\alpha=f\alpha$ is given by \begin{equation}\label{expconf} f(p)=\exp\left(\int_{0}^{1}\frac{\partial H}{\partial z}(\phi_{H}^{t}(p))dt\right).\end{equation}

\subsection{The Buhovsky-Opshtein construction}\label{bosect}

\cite[Corollary 4.4]{BO} exhibits compactly supported symplectic homeomorphisms of $\mathbb{R}^{2n}$ that map the symplectic subspace $\{(0,0)\}\times \R^{2n-2}$ to a smooth, non-symplectic submanifold---more specifically, to $\{(F(\vec{z}),0,\vec{z})|\vec{z}\in\R^{2n-2}\}$ where $F\co \R^{2n-2}\to \R$ is a continuous function whose graph is smooth and has vertical tangencies.  As we now show, Buhovsky and Opshtein's construction can be adapted to the contact context.

\begin{prop}\label{boprop}
Let $U\subset \R^{2n-1}$ be an open ball, and $F\co U\to \R$ a continuous function with compact support such that $\max_U|F|<1$.  Then for any $\delta>0$ there is a sequence of uniformly compactly supported contactomorphisms $\psi_m\co (-1,1)\times (-\delta,\delta)\times U\to (-1,1)\times (-\delta,\delta)\times U$ that converges uniformly to a homeomorphism $\psi$ of $(-1,1)\times(-\delta,\delta)\times U$ such that, for all $w\in U$, we have \[ \psi(0,0,w)=\left(F(w),0,w\right).\]  (Here the contact structure on $(-1,1)\times (-\delta,\delta)\times U$ is the kernel of $\alpha=dz-\sum_{j=1}^{n}y_jdx_j$, with $(x_1,y_1)$ the coordinates on $(-1,1)\times (-\delta,\delta)$ and $(x_2,y_2,\ldots,x_n,y_n,z)$ the coordinates on $U$.)  Moreover there is a constant $C>1$ such that the functions $f_m$ characterized by $\psi_{m}^{*}\alpha=f_m\alpha$ obey $\frac{1}{C}<\max|f_m|<C$.
\end{prop}
 
\begin{proof}
We closely follow \cite[Proof of Lemma 4.3]{BO}.  First construct a sequence of smooth functions $\{F_k\}_{k=0}^{\infty}$ on $U$ such that; \begin{itemize}\item For some compact subset $K\subset U$, each $F_k$ has support contained in $K$;
\item For some $\ep>0$, $\max_{k}\max_U|F_k|<1-\ep$;
\item $F_k\to F$ uniformly; and 
\item $F_0\equiv 0$, and $\max_U |F_k-F_{k-1}|<\frac{1}{2^k}$.\end{itemize}
Also let us abbreviate \[ G_k=F_k-F_{k-1},\quad\mbox{so }F=\sum_{k=1}^{\infty}G_k.\]

Now choose smooth functions $u,v\co \R\to \R$, with $u$ having compact support in $(-1,1)$ and $v$ having compact support in $(-\delta,\delta)$, such that: \[ u|_{[-1+\ep,1-\ep]}\equiv 1,\quad v(0)=0,\,\,v'(0)=-1 \] and, for all positive integers $k,\ell$, define \[ H_{k\ell}(x_1,y_1,x_2,\ldots,y_n,z)=u(x_1)\frac{v(\ell y_1)}{\ell}G_k(x_2,\ldots,y_n,z).\]  Let $V_{G_k}$ denote the Hamiltonian vector field of the function $G_k$ on $U$ with respect to the contact form $dz-\sum_{j=2}^{n}y_jdx_j$.  Then the Hamiltonian vector field of $H_{k\ell}$ on $(-1,1)\times(-\delta,\delta)\times U$ is  \begin{align*} X_{H_{k\ell}}=&u(x_1)\frac{v(\ell y_1)}{\ell}V_{G_k}-u(x_1)v'(\ell y_1)y_1G_k\partial_z \\ & -u(x_1)v'(\ell y_1)G_k\partial_{x_1}+\frac{v(\ell y_1)}{\ell}\left(u'(x_1)G_k+u(x_1)y_1\frac{\partial G_k}{\partial z}\right)\partial_{y_1}.\end{align*}

In particular this vector field is tangent to the hypersurface $\{y_1=0\}$, and restricts to that hypersurface as $u(x_1)G_k\partial_{x_1}$.
As in \cite{BO} the desired contactomorphisms $\psi_m$ will be given by \begin{equation}\label{psim} \psi_m=\phi_{H_{m\ell_m}}^{1}\circ\cdots\circ\phi_{H_{1\ell_1}}^{1} \end{equation} for a suitably chosen sequence $\{\ell_k\}_{k=1}^{\infty}$.  To describe the inductive procedure for choosing the $\ell_k$, note first that because all terms in the formula for $X_{H_{k\ell}}$ except the coefficient of $\partial_{x_1}$ are bounded by a $k$-dependent constant times $\frac{1}{\ell}$, and since $\max|G_k|<2^{-k}$, for all sufficiently large values of $\ell_k$ it will hold that $\max\|X_{H_{k\ell_k}}\|<C2^{-k}$ where the constant $C$ depends only on the auxiliary functions $u$ and $v$.  Also for all sufficiently large values of $\ell_k$ it will hold that \[ \max\left|\frac{\partial H_{k\ell_k}}{\partial z}\right|\leq\frac{1}{\ell_k}\max\left|\frac{\partial G_k}{\partial z}\right|<\frac{1}{k^2}.\]  Moreover since $H_{k\ell}$ has support contained in the region $\{|y_1|<\frac{\delta}{\ell}\}$ we can inductively choose the $\ell_k$ sufficiently large that, in addition to having $\max\|X_{H_{k\ell_k}}\|<C2^{-k}$ and $\max\left|\frac{\partial H_{k\ell_k}}{\partial z}\right|<\frac{1}{k^2}$, we have \begin{equation}\label{suppcond} \mathrm{supp}(\phi_{H_{k\ell_k}}^{1})\subset \left(\phi_{H_{k-1\ell_{k-1}}}^{1}\circ\cdots\circ\phi_{H_{1\ell_1}}^{1}\right)\left(\left\{|y_1|<\frac{\delta}{k}\right\}\right).\end{equation}

For such a choice of $\{\ell_k\}_{k=1}^{\infty}$, if we define $\psi_m$ as in (\ref{psim}) then (\ref{suppcond}) implies that if $y_1\neq 0$ then $\psi_m(x_1,y_1,w)$ is independent of $m$ once $m$ is sufficiently large.  On the other hand since $u|_{[-1+\ep,1-\ep]}\equiv 1$ and the restriction of $X_{H_{k\ell_k}}$ to $\{y_1=0\}$ is $u(x_1)G_k\partial_{x_1}$ we have, for all $w\in U$, \[ \psi_m(0,0,w)=\left(\sum_{k=1}^{m}G_k(w),0,w\right)=\left(F_m(w),0,w\right).\]  The estimate $\max\left|\frac{\partial H_{k\ell_k}}{\partial z}\right|<\frac{1}{k^2}$ implies, as in (\ref{expconf}), that the conformal factor of the contactomorphism $\phi_{H_{k\ell_k}}^{1}$ is bounded between $e^{-\frac{1}{k^2}}$ and $e^{\frac{1}{k^2}}$, implying an $m$-independent bound between $e^{-\frac{\pi^2}{6}}$ and $e^{\frac{\pi^2}{6}}$ for the conformal factors of the $\psi_m$.  Finally, the bound $\max\|X_{H_{k\ell_k}}\|<C2^{-k}$ implies that the sequence $\{\psi_m\}_{m=1}^{\infty}$ is uniformly Cauchy, and so uniformly converges to a map $\psi\co (-1,1)\times(-\delta,\delta)\times U\to (-1,1)\times (-\delta,\delta)\times U$.  Since $\psi_m(0,0,w)=(F_m(w),0,w)$ we indeed have $\psi(0,0,w)=(F(w),0,w)$.  That $\psi$ is injective (from which it easily follows that it is a homeomorphism since it is continuous and is the identity outside a compact subset of $(-1,1)\times (-\delta,\delta)\times U$ where $U$ is a ball) follows by the same argument that is used in \cite[Proof of Lemma 4.3]{BO}.  
\end{proof}

\begin{cor}\label{bocor}
For any $k\in\{2,\ldots,n+1\}$ and any $(2n+1)$-dimensional contact manifold $(Y,\xi)$ there exist a contact homeomorphism $\psi\co Y\to Y$, a codimension-$k$ submanifold $N\subset Y$, and a point $p\in N$ such that $N$ is not coisotropic at $p$ but $\psi(N)$ is smooth and is coisotropic at $\psi(p)$. Moreover $\psi$ can be arranged to be bounded both above and below near every point of $Y$.
\end{cor}

\begin{proof}
Choose a Darboux chart $\phi\co V\to \R^{2n+1}$ sending some point $p$ of $Y$ to the origin such that $\phi(V)$ contains $(-1,1)\times (-1,1)\times U$ for some open ball $U\subset \R^{2n-1}$, and let $N$ be a submanifold whose intersection with $V$ is identified by $\phi$ with \[ \{(x_1,y_1,x_2,\ldots,y_{n},z)\in (-1,1)\times(-1,1)\times \R^{2n-1}|x_1=y_1=0,\,y_{n-k+3}=\cdots=y_n=0\}.\] (If $k=2$ this should just be interpreted as $\{(x_1,y_1,x_2,\ldots,y_{n},z)\in (-1,1)\times(-1,1)\times \R^{2n-1}|x_1=y_1=0\}$.)  Then the $d\alpha$-orthogonal complement of $T_pN\cap \xi_p$ inside $\xi_p$ contains the tangent vector $\partial_{y_1}$, which is not contained in the tangent   
space to $N$, so $N$ is not coisotropic at $p$.  

Similarly to the proof of \cite[Corollary 4.4]{BO}, apply Proposition \ref{boprop} with $\delta=1$ and with a compactly supported function $F\co U\to (-1,1)$ 
whose graph is smooth and which restricts to a neighborhood of the origin in $\{x_2=y_2=\cdots=x_n=y_n=0\}$ as a function $f$ of the single variable $z$ with $f(0)=0$, such that $f$ is invertible on a neighborhood of $0$ on which $f^{-1}$ is smooth with $(f^{-1})'(0)=0$.  (So $f$ itself has derivative tending to $\pm\infty$ at $0$.)  The resulting contact homeomorphism $\psi$ will have $\psi(p)=p$ and will send $N$ to a smooth submanifold whose intersection with $V$ is contained in the hypersurface $\{y_1=0\}$ and coincides there with the graph of $F$.  The tangent space $T_{p}\psi(N)$ will be spanned by $\partial_{x_1},\ldots,\partial_{x_n}$ together with some subset (depending on $k$) of the $\partial_{y_j}$ with $j\geq 2$; in particular this tangent space will be a coisotropic subspace of $\xi_{\psi(p)}$, and so $\psi(N)$ is coisotropic at $p=\psi(p)$.  That $\psi$ is bounded both above and below follows directly from the last sentence of Proposition \ref{boprop}.
\end{proof}

\subsection{Collapsing toward a Legendrian torus}
\label{accel}

In this section we describe a family of examples of contact homeomorphisms $\psi$ of a neighborhood of a Legendrian torus which are not bounded below near points on the torus, and which can be arranged to send a nowhere-Legendrian submanifold $N$ to a smooth submanifold one that is tangent, possibly (depending one one's choice of parameters) even to infinite order, to the Legendrian torus; moreover unlike in Section \ref{bosect} the restriction $\psi|_N$ can be arranged to be a smooth map.  The section concludes with the proof of Theorem \ref{exprop}.
 
Work throughout this section in a smooth manifold of the form $\mathcal{B}_{V}=(\R^n/\Z^n)\times V$ where $V$ is a neighborhood of the origin in $\R^{n+1}$ (which will be specified more precisely in particular examples), with coordinates $\vec{x}=(x_1,\ldots,x_n)$ on $\R^n/\Z^n$ and $(\vec{y},z)=(y_1,\ldots,y_n,z)$ on $V$.  We continue to use the contact form $\alpha=dz-\sum_jy_jdx_j$ on $\mathcal{B}_V$.  By the Legendrian neighborhood theorem any Legendrian torus $T$ in a contact manifold has a neighborhood contactomorphic to such a contact manifold $(\mathcal{B}_V,\ker\alpha)$, so the constructions of this section can be exported to other contact manifolds.   

Let \[ \mathcal{Z}=\{(x_1,\ldots,x_n,y_1,\ldots,y_n,z)\in\mathcal{B}_V|y_1=\cdots=y_n=z=0\}\] and \[ \mathcal{B}_{V}^*=\mathcal{B}_{V}\setminus \mathcal{Z}.\]  
We will consider flows of autonomous contact Hamiltonians $H\co \mathcal{B}_{V}^{*}\to \R$ that extend continuously to all of $\mathcal{B}_V$.
Our open sets $V$ and Hamiltonians $H=H_{F,\rho}$ are prescribed as follows:
\begin{itemize} \item We have $V=\rho^{-1}([0,c))$ for some $c>0$ where $\rho\co \R^{n+1}\to [0,\infty)$ is a proper smooth function of the form \[ \rho(\vec{y},z)=\rho_y(\vec{y})+z^{d_z} \] where $\rho_y$ obeys $\rho_y(t\vec{y})=t^{d_y}\rho_y(\vec{y})$ and $d_y,d_z$ are even integers with $d_y\geq d_z\geq 2$.    (More specifically, in our examples we will take either $\rho(\vec{y},z)=z^2+\sum_j y_{j}^{2}$ or $\rho(\vec{y},z)=z^2+\sum_j y_{j}^{4}$.)
\item There is a smooth function $F\co (-\log c,\infty)\to (-\infty,0]$ such that $H\co \mathcal{B}_{V}^{*}\to \R$ is given by \[ H_{F,\rho} (\vec{x},\vec{y},z)=zF(-\log \rho(\vec{y},z)).\]
\end{itemize}

We also assume that $F$ satisfies the following throughout the rest of this section:

\begin{assm}\label{fass} \hspace{2em}\begin{itemize}
\item[(i)] $F'\leq 0$ everywhere. 
\item[(ii)] There is $u_0>-\log c$ such that $F(u)=0$ if and only if $u\leq u_0$.
\item[(iii)] For one and hence every $u_1>u_0$ we have \begin{equation}\label{nonint} 
\int_{u_1}^{\infty}\frac{du}{F(u)}=-\infty. \end{equation} 
\item[(iv)]  \[ \lim_{u\to\infty} e^{(\frac{1}{d_y}-\frac{1}{d_z})u}F'(u)=0.\]
\item[(v)] \[ \lim_{u\to\infty}\frac{F'(u)}{F(u)}=0.\]
\end{itemize}
\end{assm}

More specifically, in the examples at the end of this section we will take $d_y=d_z=2$ and $F(u)=-u^{\beta}$ for sufficiently large $u$ and some number $\beta$ with  $0<\beta<1$, or $d_y=4$ and $d_z=2$ and $F(u)$ equal either to $-u$ or to $-u\log u$ for sufficiently large $u$. 

Trajectories of the Hamiltonian flow of such a function $H_{F,\rho}$ are then, in view of (\ref{hamform}), solutions to the following system: \begin{align}\label{zfflow}
x'_{j}&=\frac{z}{\rho(\vec{y},z)}\frac{\partial \rho}{\partial y_j}F'(-\log\rho(\vec{y},z)) \nonumber \\
y'_{j}&=y_{j}F(-\log \rho(\vec{y},z))-\frac{y_jz}{\rho}\frac{\partial \rho}{\partial z}F'(-\log\rho(\vec{y},z))\\
z' &= zF(-\log\rho(\vec{y},z))+\sum_j\frac{y_j z}{\rho}\frac{\partial \rho}{\partial y_j}F'(-\log\rho(\vec{y},z))\nonumber.\end{align}

In particular, such solutions always obey \begin{align*} \frac{d}{dt}\left(\rho(\vec{y}(t),z(t))\right)&=\sum_{j}\frac{\partial \rho}{\partial y_j}y'_j+\frac{\partial\rho}{\partial z}z' &= \left(\sum_jy_j\frac{\partial \rho}{\partial y_j} + z\frac{\partial \rho}{\partial z}\right)F(-\log\rho(\vec{y},z)).\end{align*}  (Note the convenient cancellation of the terms involving $F'$.)  By Euler's homogeneous function theorem  one has $\sum_jy_j\frac{\partial \rho}{\partial y_j}=d_y\rho_y$, and  obviously $z\frac{\partial\rho}{\partial z}=d_zz^{d_z}$ and so (bearing in mind that  $F\leq 0$ and $d_y\geq d_z$) \[ d_y\rho(\vec{y},z)F(-\log\rho(\vec{y},z))\leq 
\frac{d}{dt}\left(\rho(\vec{y},z)\right)\leq d_z\rho(\vec{y},z)F(-\log\rho(\vec{y},z)),\] \emph{i.e.}, \begin{equation}\label{dfineq} -d_zF(-\log\rho(\vec{y},z))\leq \frac{d}{dt}\left(-\log(\rho(\vec{y},z))\right)\leq -d_yF(-\log\rho(\vec{y},z) )\end{equation} for any flowline $t\mapsto (\vec{x}(t),\vec{y}(t),z(t))$ of the Hamiltonian flow of $H_{F,\rho}$.  

We will see presently that Assumptions \ref{fass} together with (\ref{dfineq}) imply that Hamiltonian flowlines for $H_{F,\rho} $ which begin in $\mathcal{B}_{V}^{*}$ at $t=0$ exist (within $\mathcal{B}_{V}^{*}$) for all positive $t$.    By (\ref{nonint}), the map $G\co (u_0,\infty)\to (-\infty,\infty)$ defined by \[ G(u)=\int_{u_1}^{u}\frac{dv}{F(v)} \] (for an arbitrary choice of $u_1>u_0$) is a diffeomorphism, with $G'(u)=\frac{1}{F(u)}$.  (That $G(u)\to\infty$ as $u\to u_{0}^{+}$ follows from the fact that $F$ vanishes to infinite order at $u_0$.) If $r\in \R$ and $v_0>u_0$, the unique solution to the equation $u'(t)=-rF(u(t))$ obeying an initial condition $u(0)=v_0$ is then $u(t)=G^{-1}(G(v_0)-rt)$.  In particular this solution exists and remains in the interval $(u_0,\infty)$ for all time.  Of course if we instead have $v_0\leq u_0$ the unique solution to $u'=-rF(u)$ with $u(0)=v_0$ is constant.

In the case that $d_y=d_z$, then based on (\ref{dfineq}) the above considerations allow one to compute $\rho(\vec{y}(t),z(t))$ as a function of $t$ directly from $F$.  More generally we have the following:
\begin{prop}\label{bihari}
With $F$ and $G$ as above, suppose that $I$ is an open interval around zero and $u\co I\to \R$ obeys the differential inequalities \begin{equation}\label{dzdy} -d_zF(u(t))\leq u'(t)\leq -d_yF(u(t)),\end{equation} and that $u(0)>u_0$.  Then for all $t\in I$ with $t\geq 0$, \[ G^{-1}(G(u(0))-d_zt)\leq u(t)\leq G^{-1}(G(u(0))-d_yt).\] 
\end{prop}

\begin{proof} Since $-d_zF(v)\geq 0$ for all $v$ the hypothesis implies that $u$ is a monotone increasing function and hence in particular that $u(t)>u_0$ and hence $F(u(t))<0$ for all $t\geq 0$.
We have \[ \frac{d}{dt}G(u(t))=G'(u(t))u'(t)=\frac{u'(t)}{F(u(t))}\in [-d_y,-d_z] \mbox{ for all $t$} \]  based on (\ref{dzdy}) and the fact that $F(u(t))< 0$.  Integrating with respect to $t$ shows that, if $t\geq 0$, then \[ G(u(0))-d_yt\leq G(u(t))\leq G(u(0))-d_zt.\]  Since $G$ and hence also $G^{-1}$ is a decreasing function, the above inequalities directly imply that $G^{-1}(G(u(0))-d_zt)\leq u(t)\leq G^{-1}(G(u(0))-d_yt)$.  
\end{proof}

\begin{cor}
If $F\co (-\log c,\infty)\to (-\infty,0]$ satisfies Assumptions \ref{fass} then the contact Hamiltonian flow $\phi_{H_{F,\rho}}^{t}$ of $H_{F,\rho} \co \mathcal{B}_{V}^{*}\to \R$ is well-defined as a diffeomorphism of $\mathcal{B}_{V}^{*}$ for all $t\in \R$, and is the identity on the subset of $\mathcal{B}_{V}^{*}$ on which $\rho(\vec{y},z)> e^{-u_0}$. 
\end{cor}

\begin{proof}
Since $H_{F,\rho}$ is smooth throughout $\mathcal{B}_{V}^{*}$, standard results in ODE theory imply that in order for the corollary to be false there would need to be an integral curve $\gamma\co (T_-,T_+)\to \mathcal{B}_{V}^{*}$ (with $T_-,T_+$ both finite and $T_-<0<T_+$) of $X_{H_{F,\rho}}$  whose image is not contained in any compact subset of $\mathcal{B}_{V}^{*}$.  Now since $H_{F,\rho}$ vanishes everywhere that $\rho(\vec{y},z)\in [e^{-u_0},c)$, an integral curve of $X_{H_{F,\rho}}$ must either be constant or be contained in the region $\{0<\rho(\vec{y},z)\leq e^{-u_0}\}$.   By (\ref{dfineq}), the function $(\vec{x},\vec{y},z)\mapsto -\log\rho(\vec{y},z)$ is monototone increasing along the integral curve $\gamma$, and by Proposition \ref{bihari} if the value of this function at time zero is $v_0>-\log c$ then it will never take a value larger than $G^{-1}(G(v_0)-d_yT_+)$ for $t\in [T_-,T_+]$.  So in this case $\rho(\vec{y},z)\geq e^{-  G^{-1}(G(v_0)-d_yT_+)}$ everywhere along $\gamma$.  Thus every integral curve of $X_{H_{F,\rho}}$  defined on a bounded time interval remains inside a compact subset of  $\mathcal{B}_{V}^{*}$, as desired.
\end{proof}

\begin{prop}\label{chomeo}  Let $F$ satisfy Assumptions \ref{fass}. Then for all $t\in\R$ the time-$t$ map $\phi_{H_{F,\rho}}^{t}\co \mathcal{B}_{V}^{*}\to \mathcal{B}_{V}^{*}$ extends by the identity along the zero section $\mathcal{Z}=\{y_1=\cdots=y_n=z=0\}$ to a contact homeomorphism of $\mathcal{B}_V$, which we denote by $\overline{\phi}_{H_{F,\rho}}^{t}$.  

Moreover, assuming that $t>0$ and that $\lim_{u\to\infty}F(u)=-\infty$, $\overline{\phi}_{H_{F,\rho}}^{t}$ is the $C^0$-limit of contactomorphisms $\psi_m\co \mathcal{B}_V\to\mathcal{B}_V$ having uniform compact support, and such that $\psi_{m}^{*}\alpha=f_m\alpha$ where the smooth functions $f_m$ uniformly converge to a continuous function $f\co \mathcal{B}_{V}\to [0,\infty)$ with $\phi_{H_{F,\rho}}^{t*}\alpha=f\alpha$ on $\mathcal{B}_{V}^{*}$ and $f|_\mathcal{Z}=0$.
\end{prop}

\begin{proof} Since the inverse of a contact homeomorphism is a contact homeomorphism it suffices to prove the result for $t>0$.  Since $\phi_{tH}^{1}=\phi_{H}^{t}$, by replacing $F$ by $tF$ (which does not affect whether $F$ obeys the hypotheses for $t>0$) we may as well assume that $t=1$.  

Let us first show that the map $\overline{\phi}_{H_{F,\rho}}^{1}$ given by extending $\phi_{H_{F,\rho}}^{1}$ by the identity over the zero section is a homeomorphism. Proposition \ref{bihari} shows that if $\ep>0$ there is $\delta>0$ such that $\overline{\phi}_{H_{F,\rho}}^{1}$ and its inverse each map the region $\{\rho(\vec{y},z)<\delta\}$ inside the region $\{\rho(\vec{y},z)<\epsilon\}$, so in order to establish the continuity of $\overline{\phi}_{H_{F,\rho}}^{1}$  and of its inverse  along the zero section it remains only to establish that the $x_j$ coordinates of these maps are continuous along the zero section.  For this purpose it suffices to check that the $\partial_{x_j}$-components of the Hamiltonian vector field $X_{H_{F,\rho}}$ can be bounded in terms of a function of $\rho(\vec{y},z)$ that approaches zero as $\rho(\vec{y},z)\to 0$.  Write $g_j$ for the $\partial_{x_j}$-component of $X_{H_{F,\rho} }$; thus \[ g_j(\vec{x},\vec{y},z)=\frac{z}{\rho(\vec{y},z)}\frac{\partial\rho}{\partial y_j}(\vec{y},z)F'(-\log(\rho(\vec{y},z))).\]   Our hypotheses on $\rho$ imply that we have $\rho(t^{d_z}\vec{y},t^{d_y}z)=t^{d_yd_z}\rho(\vec{y},z)$, and that $\frac{\partial \rho}{\partial y_j}$ is independent of $z$ and obeys $\frac{\partial \rho}{\partial y_j}(t\vec{y},z)=t^{d_{y}-1}\frac{\partial \rho}{\partial y_j}(\vec{y},z)$.  Given $(\vec{y},z)\in \R^{n+1}\setminus\{(\vec{0},0)\}$, we can find $t>0$ and $(\vec{y}_0,z_0)$ such that $\rho(\vec{y}_0,z_0)=1$ and $(t^{d_z}\vec{y}_0,t^{d_y}z_0)=(\vec{y},z)$; we then have $\rho(\vec{y},z)=t^{d_yd_z}$ and \begin{align*} |g_j(\vec{x},\vec{y},z)|&=\left|\frac{t^{d_y}z_0}{t^{d_yd_z}}(t^{d_z})^{d_{y}-1}\frac{\partial \rho}{\partial y_j}(\vec{y}_0,z_0)F'(-\log\rho(\vec{y},z))\right|\leq M_jt^{d_y-d_z}|F'(-\log\rho(\vec{y},z)) | \\&= M_j \rho(\vec{y},z)^{\frac{1}{d_z}-\frac{1}{d_y}}|F'(-\log\rho(\vec{y},z))|
\end{align*}  where $M_j$ is the maximal value of $z\frac{\partial \rho}{\partial y_j}$ on $\{\rho(\vec{y},z)=1\}$.  So Assumption \ref{fass}(iv) implies that $|g_j|$ is bounded above by a function of $\rho(\vec{y},z)$ that approaches zero as $\rho\to 0$, which as noted earlier suffices to establish that $\overline{\phi}_{H_{F,\rho}}^{1}$ is a homeomorphism.

We will now exhibit a sequence of contactomorphisms that uniformly converges to $\overline{\phi}_{H_{F,\rho}}^{1}$. As before let $G(u)=\int_{u_1}^{u}\frac{dv}{F(v)}$ and, given a sufficiently large $m\in \N$, choose a smooth function $\beta_m\co \R\to \R$ such that: \begin{itemize} \item $\beta_m(u)=u$ for all $u\leq G^{-1}(G(m)-d_y)$; \item $0\leq \beta'_m(u)\leq 1$ for all $u$; and \item $\beta'_m(u)=0$ for all $u\geq 1+ G^{-1}(G(m)-d_y)$.\end{itemize}  Set $F_m=F\circ \beta_m$. Then $F_m$ also satisfies Assumptions \ref{fass}, and it has the additional property that there are constants $u_m,c_m$ such that $F_m(u)=c_m$ for all $u>u_m$.  The latter property immediately implies that $H_{F_m,\rho}(\vec{x},\vec{y},z)=zF_m(-\log\rho(\vec{y},z))$ extends smoothly across the zero section, and hence so too does its Hamiltonian flow $\phi_{H_{F_m,\rho}}^{t}$ (specifically this flow is given on a neighborhood of the zero section by $\phi_{H_{F_m,\rho}}^{t}(\vec{x},\vec{y},z)=(\vec{x},e^{c_mt}\vec{y},e^{c_mt}z)$). 

The contactomorphisms $\psi_m$ promised in the statement of the proposition are given by $\psi_m=\phi_{H_{F_m,\rho}}^{1}$.  Clearly these are all supported in the compact set on which $\rho(\vec{y},z)\leq e^{-u_0}$.  Let us show that $\phi_{H_{F_m,\rho}}^{1}\to \overline{\phi}_{H_{F,\rho}}^{1}$ uniformly.  By (\ref{dfineq}) and Proposition \ref{bihari}, if $-\log\rho(\vec{y},z)\leq m$ then for all $t\in [0,1]$, writing $(\vec{x}(t),\vec{y}(t),z(t))=\phi_{H_{F,\rho}}^{t}(\vec{x},\vec{y},z)$,  we will have $-\log\rho(\vec{y}(t),z(t))\leq G^{-1}(G(m)-d_y)$.  Since $H_{F,\rho}$ coincides with $H_{F_m,\rho}$ everywhere that $-\log \rho(\vec{y},z)\leq G^{-1}(G(m)-d_y)$ it follows that the restriction of $\phi_{H_{F_m,\rho}}^{1}$ to $\{\rho(\vec{y},z)\geq e^{-m}\}$ coincides with that of $\overline{\phi}_{H_{F,\rho}}^{1}$.  Of course $\overline{\phi}_{H_{F,\rho}}^{1}$ and $\phi_{H_{F_m,\rho}}^{1}$ also coincide on the zero section.  Moreover the same analysis that was used in the proof that $\overline{\phi}_{H_{F,\rho}}^{1}$ is a homeomorphism shows that the Hamiltonian vector fields of $H_{F,\rho}$ and $H_{F_m,\rho}$ have norms that are uniformly bounded by a function of $\rho(\vec{y},z)$ that converges to zero as $\rho\to 0$.  From this and the fact that $\rho(\vec{y},z)$ decreases along the Hamiltonian flows of $H_{F_m,\rho}$ and of $H_{F,\rho}$ it readily follows that, for any $\ep>0$, there is $\rho_{\ep}$ such that the restrictions of $\phi_{H_{F,\rho}}^{1}$ and $\phi_{H_{F_m,\rho}}^{1}$ to $\{0< \rho(\vec{y},z)\leq \rho_{\ep}\}$ all have  $C^0$-distance at most $\frac{\ep}{2}$ from the identity, and hence at most $\ep$ from each other.    So once $m$ is so large that $e^{-m}<\rho_{\ep}$ we see that $\psi_m=\phi_{H_{F_m,\rho}}^{1}$ is within $C^0$-distance $\ep$ of $\overline{\phi}_{H_{F,\rho}}^{1}$ throughout $\mathcal{B}_V$. 

It remains only to prove the statement at the end of the proposition about the conformal factors $f_m$ of the $\psi_m$.  These conformal factors are related by (\ref{expconf}) to the functions $\frac{\partial H_{F_m,\rho}}{\partial z}$.  By construction, the maps $\frac{\partial H_{F,\rho}}{\partial z}\circ \phi_{H_{F,\rho}}^{t}$ and $\frac{\partial H_{F_m,\rho}}{\partial z}\circ \phi_{H_{F_m,\rho}}^{t}$ coincide on the set $\{\rho(\vec{y},z)\geq e^{-m}\}$ for all $t\in [0,1]$, so we have \begin{equation}\label{fmf} f_m|_{\{\rho(\vec{y},z)\geq e^{-m}\}}=f|_{\{\rho(\vec{y},z)\geq e^{-m}\}} \end{equation} where as is the statement of the proposition $f\co\mathcal{B}_V\to \R$ restricts to $\mathcal{B}_{V}^{*}$ as the conformal factor of $\phi_{H_{F,\rho}}^{1}$ and to $\mathcal{Z}$ as zero. 

Now \[ \frac{\partial H_{F_m,\rho}}{\partial z}=F(\beta_m(-\log\rho(\vec{y},z)))-\frac{d_zz^{d_z}}{\rho(\vec{y},z)}\beta'_m(-\log\rho(\vec{y},z))F'(\beta_m(-\log\rho(\vec{y},z))).\] So if $m$ is large enough Assumptions \ref{fass}(v) and (i) imply that, if $\rho(\vec{y},z)\leq e^{-m}$, then \[ \frac{\partial H_{F_m,\rho}}{\partial z}(\vec{y},z)\leq \frac{1}{2}F(\beta_m(-\log\rho(\vec{y},z)))\leq \frac{1}{2}F(m).\]  So since the set $\{\rho(\vec{y},z)\leq e^{-m}\}$ is preserved by $\phi_{H_{F_m,\rho}}^{t}$ for $t\geq 0$ it follows from (\ref{expconf}) that \[ f_m|_{\{\rho(\vec{y},z)\leq e^{-m}\}}\leq e^{\frac{1}{2}F(m)}.\]  The same reasoning applied to $F$ in place of $F_m$ shows that $f|_{\{\rho(\vec{y},z)\leq e^{-m}\}}\leq e^{\frac{1}{2}F(m)}$.  Of course $f_m$ and $f$ are both nonnegative, so in view of (\ref{fmf}) we see that \[ \sup_{\mathcal{B}_V}|f_m-f|\leq e^{\frac{1}{2}F(m)},\] which converges to zero based on our assumption that $\lim_{u\to\infty}F(u)=-\infty$.  So indeed $f_m\to f$ uniformly, and hence $f$ is continuous.\end{proof}

\begin{cor}\label{nonbound} Assuming that $t>0$ and $\lim_{u\to\infty}F(u)=-\infty$, the contact homeomorphism $\overline{\phi}_{H_{F,\rho}}^{t}$ is bounded above near every point of $\mathcal{B}_V$, but for every $p\in \mathcal{Z}$ it is not bounded below near $p$.
\end{cor}

\begin{proof}
The functions $f_m$ in Proposition \ref{chomeo} are positive and uniformly bounded above (since they converge uniformly to the function $f$, which is bounded above since it is continuous on $\mathcal{B}_V$ and equal to $1$ outside a compact subset of $\mathcal{B}_V$); this suffices to prove that $\overline{\phi}_{H_{F,\rho}}^{t}$ is  bounded above near every point. 

If $p\in\mathcal{Z}$ we can see that $\overline{\phi}_{H_{F,\rho}}^{t}$ is not bounded below near $p$ by using Propositions \ref{bdcrit} and \ref{chomeo}. Indeed the former implies that if $\overline{\phi}_{H_{F,\rho}}^{t}$ were bounded below near $p$ there would be $\delta>0$ so that for every sufficiently small neighborhood $W$ of $p$ we would have \[ \int_{\overline{\phi}_{H_{F,\rho}}^{t}(W)}\alpha\wedge(d\alpha)^{\wedge n}\geq \delta\int_{W}\alpha\wedge (d\alpha)^{\wedge n}.\]  But given any $\delta>0$, if we choose $W$ so small that the function $f$ in Proposition \ref{chomeo} has $\sup_W|f|^{n+1}<\delta$ we see, using that $\overline{\phi}_{H_{F,\rho}}^{t}$ is smooth on the full-measure subset $\mathcal{B}_{V}^{*}\subset \mathcal{B}_V$ (allowing us to apply the change of variables theorem), \[ 
\int_{\overline{\phi}_{H_{F,\rho}}^{t}(W)}\alpha\wedge(d\alpha)^{\wedge n}=\int_{\overline{\phi}_{H_{F,\rho}}^{t}(W\cap\mathcal{B}_{V}^{*})}\alpha\wedge(d\alpha)^{\wedge n}=\int_{W\cap\mathcal{B}_{V}^{*}}f^{n+1}\alpha\wedge (d\alpha)^{\wedge n}<\delta\int_W\alpha\wedge (d\alpha)^{\wedge n},\] a contradiction.
\end{proof}

\begin{prop} \label{wall} The contact homeomorphism $\overline{\phi}_{H_{F,\rho}}^{t}\co\mathcal{B}_V\to\mathcal{B}_V$ from Proposition \ref{chomeo} has restriction to the locus $\mathcal{W}:=\{(\vec{x},\vec{0},z)| 0<|z|<e^{-u_0/d_z}\}\subset \mathcal{B}_{V}^{*}$  given by \[ \overline{\phi}_{H_{F,\rho}}^{1}(\vec{x},\vec{0},z)=\left(\vec{x},\vec{0},e^{-\frac{1}{d_z}G^{-1}(G(-d_z\log |z|)-d_zt)}\mathrm{sgn}(z)\right)\] where $G\co (u_0,\infty)\to(-\infty,\infty)$ is an antiderivative of $\frac{1}{F}$.  Moreover if $t>0$ and $\lim_{u\to\infty}F(u)=-\infty$, the contactomorphisms $\psi_m$ from the proof of Proposition \ref{chomeo} have the property that $\psi_m|_{\mathcal{W}}\to \overline{\phi}_{H_{F,\rho}}^{t}|_{\mathcal{W}}$ in the $C^1$ topology.
\end{prop}

\begin{proof} 
Examining (\ref{zfflow}) and recalling that $\rho(\vec{0},z)=z^{d_z}$ where $d_z$ is an even integer and that  $\frac{\partial \rho}{\partial y_j}$ is homogeneous of degree $d_y-1\geq 1$ in $\vec{y}$ and hence vanishes where $\vec{y}=0$, we see that one obtains integral curves of $X_{H_{F,\rho}}$ by taking each $x_j$ equal to an arbitrary constant, each $y_j$ equal to zero, and each $z$ equal to a solution of \[ z'=zF(-\log z^{d_z})=zF(-d_z\log |z|) .\]  The latter equation can be rewritten as $\frac{d}{dt}\log|z(t)|=F(-d_z\log|z(t)|)$, which has general solution $-\log|z(t)|=\frac{1}{d_z}G^{-1}(G(-d_z\log|z(0)|)-d_zt)$. Since $\overline{\phi}_{H_{F,\rho}}^{t}$ is given on $\mathcal{B}_{V}^{*}$ as the time-$t$ flow of $X_{H_{F,\rho}}$, the formula in the statement of the proposition follows directly by exponentiating this formula for $\log|z(t)|$.

For the second statement, recall that the $\psi_m$ are taken in the proof of Proposition \ref{chomeo} to be of the form $\psi_m=\phi_{H_{F_m,\rho}}^{t}$ where the $F_m$ satisfy Assumptions \ref{fass} and additionally have $F_m|_{[u_m,\infty)}$ constant for suitable $u_m$ (so that $H_{F_m,\rho}$ can be seen as a smooth function on all of $\mathcal{B}_V$).  So, as an instance of the first statement of this proposition, we have $\psi_m(\vec{x},\vec{0},z)=(\vec{x},0,g_m(z))$ for a certain smooth function $g_m$. Let us write $g(z)=e^{-\frac{1}{d_z}G^{-1}(G(-d_z\log |z|)-d_zt)}$ for the third component of $\overline{\phi}_{H_{F,\rho}}^{t}|_{\mathcal{W}}$.  Since $\psi_m\to \overline{\phi}_{H_{F,\rho}}^{t}$ uniformly we  evidently have $g_m\to g$ uniformly.  Furthermore, notice that the contact form $\alpha=dz-\sum_jy_jdx_j$ restricts to $\mathcal{W}$ as $dz$.  So  $\psi_{m}^{*}(\alpha|_{\mathcal{W}})=g'_m(z)dz=g'_m(z)\alpha|_{\mathcal{W}}$ and likewise $\overline{\phi}_{H_{F,\rho}}^{t*}(\alpha|_{\mathcal{W}})=g'(z)\alpha|_{\mathcal{W}}$.  So the last clause of Proposition \ref{chomeo} implies that $g'_m\to g'$ uniformly, and hence $g_{m}\to g$ in $C^1$.  
\end{proof}

\begin{ex}\label{square}
If $\rho(\vec{y},z)=\sum_jy_{j}^{2}+z^2$ it turns out that one can give an explicit formula for $\phi_{H_{F,\rho}}^{t}$ on all of $\mathcal{B}_{V}^{*}$, not just on the locus where $\vec{y}=\vec{0}$.  Specifically, letting as before $G$ be an antiderivative of $\frac{1}{F}$, and also abbreviating \[ u_t(\vec{y},z)=G^{-1}(G(-\log\rho(\vec{y},z))-2t),\] one has $\phi_{H_{F,\rho}}^{t}(\vec{x},\vec{y},z)=(\vec{X}(t),\vec{Y}(t),Z(t))$ where:\begin{align*}  \vec{X}(t)&=\vec{x}+ \left(\arctan\left(\frac{\|\vec{y}\|}{z}\right)-\arctan\left(\frac{F(u_t(\vec{y},z))}{F(u_0(\vec{y},z))}\frac{\|\vec{y}\|}{z}\right)\right)\frac{\vec{y}}{\|\vec{y}\|} ,\\ \vec{Y}(t)&=-\frac{e^{-u_t(\vec{y},z)/2}F(u_t(\vec{y},z))\vec{y}}{\sqrt{\|F(u_t(\vec{y},z)\vec{y}\|^2+(F(u_0(\vec{y},z))z)^2} },\\ Z(t)&=-\frac{e^{-u_t(\vec{y},z)/2}F(u_0(\vec{y},z))z}{\sqrt{\|F(u_t(\vec{y},z)\vec{y}\|^2+(F(u_0(\vec{y},z))z)^2} }.\end{align*}

(To derive such a formula from scratch, one can observe that that, along the Hamiltonian flow of $H_{F,\rho}$, one has $\rho(\vec{Y}(t),Z(t))=u_t(\vec{y},z)$ by Proposition \ref{bihari} since $d_y=d_z=2$, and moreover that if $\gamma(t)=\frac{\|\vec{Y}(t)\|^2-Z(t)^2}{\|\vec{Y}(t)\|^2+Z(t)^2}$, then one has $\gamma'(t)=-2(1-\gamma(t)^2)F'(-\log u_t(\vec{y},z))$, which one can then solve easily for $\gamma$, hence determining $\|\vec{Y}(t)\|$ and $Z(t)$. Of course if one has been given the above formulas one can also simply confirm by direct substitution that they satisfy the ODEs (\ref{zfflow}).) 

Specializing this example further, we could choose $F$ so that $F(v)=-\sqrt{v}$ for all $v\geq v_1$ (for an arbitrary $v_1$ which is greater than $-\log c$ in the notation of Assumption \ref{fass}). This yields, for all $\vec{y},z$ with $\rho(\vec{y},z)\leq e^{-v_1}$, \[ u_t(\vec{y},z)=   \left(\sqrt{-\log(\|\vec{y}\|^2+z^2)}+t\right)^2.\] We find in particular that \[ \overline{\phi}_{H_{F,\rho}}^{t}(\vec{x},\vec{0},z)=\left(\vec{x},\vec{0},e^{-t\sqrt{-2\log|z|}-t^2/2}z\right)\] for all sufficiently small $z$.  For any fixed $t>0$ the third component above is a $C^1$ function of $z$ which vanishes together with its derivative at $z=0$; however its second derivative at $z=0$ does not exist.  Thus, for $t>0$, $\overline{\phi}_{H_{F,\rho}}^{t}$ maps a neighborhood of the origin in the nowhere-Legendrian submanifold $\{(\vec{x},\vec{0},x_n)\}$ of $\mathcal{B}_V$ to a neighborhood of the origin in the $C^1$-submanifold $\{(\vec{x},\vec{0},e^{-t\sqrt{-2\log|x_n|}-t^2/2}x_n)\}$, which is tangent to the contact distribution at the origin.  

One obtains similar behavior if one takes $F(v)=-v^{\beta}$ with $0<\beta<1$.  Note that the condition $\beta<1$ is forced by Assumption \ref{fass}(iv)  because in this example $d_y=d_z$.
\end{ex}

\begin{ex}\label{fourfinite}
If we instead take $\rho(\vec{y},z)=\sum_{j}y_{j}^{4}+z^2$ then Assumption \ref{fass}(iv) only requires that $\lim_{u\to\infty}e^{-u/4}F'(u)=0$, allowing more freedom in the choice of $F$ and ultimately leading to examples that improve on the $C^1$-smoothness in Example \ref{square}. (In Example \ref{square}, Assumption \ref{fass}(iv) required $\lim_{u\to\infty}F'(u)=0$.)  With this new choice of $\rho$, the author does not know an explicit formula for the maps $\phi_{H_{F,\rho}}^{t}$ on all of $\mathcal{B}_{V}^{*}$ as in Example \ref{square}, but Proposition \ref{wall} still applies to compute their restrictions to the locus $\{\vec{y}=\vec{0}\}$.  

More concretely, if $F(v)=-v$ for all sufficiently large $v$, one finds (for $|z|$ sufficiently small) \[ G^{-1}(G(-2\log|z|)-2t)=e^{2t+\log(-2\log|z|)}=-2e^{2t}\log|z|, \] so that Proposition \ref{wall} gives \[ \phi_{H}^{t}(\vec{x},\vec{0},z)=\left(\vec{x},\vec{0},\mathrm{sgn}(z)|z|^{e^{2t}}\right).\]  So for any odd integer $m>1$, the contact homeomorphism $\overline{\phi}_{H_{F,\rho}}^{\frac{1}{2}\log m}$ maps a neighborhood of the origin in the nowhere-Legendrian submanifold  $\{(\vec{x},\vec{0},x_n)\}\subset\mathcal{B}_V$ to a neighborhood of the origin in $\{(\vec{x},0,x_{n}^{m})\}$, which is of course smooth and has an order-$m$ tangency to the contact distribution at the origin.
\end{ex}

\begin{ex}\label{fourinf}
To get an infinite-order tangency, we can again take $\rho(\vec{y},z)=\sum_jy_{j}^{4}+z^2$ and now set $F(v)=-v\log v$ for all sufficiently large $v$, so that $\frac{1}{F}$ has antiderivative $G(v)=-\log(\log v)$ for all large $v$. One then computes that $G^{-1}(G(-2\log|z|)-2t)=\left(-2\log|z|\right)^{e^{2t}}$ and hence that, by Proposition \ref{wall}, \[ \phi_{H_{F,\rho}}^{t}(\vec{x},\vec{0},z)=\left(\vec{x},\vec{0},\mathrm{sgn}(z)e^{-\frac{1}{2}\left(\log \frac{1}{z^2}\right)^{e^{2t}}}\right)\] for all sufficiently small $z$.  For any fixed $t>0$ and any positive integer $m$ the third component above approaches zero as $z\to 0$ faster than $|z|^{m}=e^{-\frac{m}{2}\log\frac{1}{z^2}}$.  Consequently a neighborhood of the origin in the nowhere-Legendrian submanifold $\{\vec{y}=\vec{0},z=x_n\}$ is sent by $\overline{\phi}_{H_{F,\rho}}^{t}$ to a smooth submanifold with an infinite-order tangency to the contact disribution at the origin. 
\end{ex} 

\begin{proof}[Proof of Theorem \ref{exprop}]  Corollary \ref{bocor}, specialized to the case $k=n+1$, gives examples for variation (i) of the theorem. Either Example \ref{fourfinite} or Example \ref{fourinf} supplies instances of variation (iii), using Corollary \ref{nonbound} and  bearing in mind that an arbitrary contact manifold contains Legendrian tori (contained in Darboux charts, for instance) which have tubular neighborhoods contactomorphic to $\mathcal{B}_V$, and that our examples are limits of contactomorphisms that are uniformly compactly supported in $\mathcal{B}_V$ which can thus be exported to any contact manifold. 

 We now explain how to combine the constructions in Section \ref{bosect} and in the current section to provide examples for variation (ii). First take a standard neighborhood $\mathcal{N}\cong \mathcal{B}_V$ of a Legendrian torus in $(Y,\xi)$ and let $\psi_3\co Y\to Y$ be given by one of the contact homeomorphisms from Example \ref{fourinf} within $\mathcal{N}$ and by the identity outside $\mathcal{N}$. Next let $\psi_1$ be a contact homeomorphism of $Y$ given by the identity outside of a small neighborhood $W=(-\delta,\delta)^{2n+1}\subset R^n/\Z^n\times \R^{n+1}$ of the origin under the identification $\mathcal{B}_V\cong\mathcal{N}\subset Y$ and, inside $W$, by a contact homeomorphism as in Proposition \ref{boprop} for a function $F$ supported inside $(-\delta,\delta)^{2n-1}$, with $\max|F|<\delta$, and such that $F$ has (as in Corollary \ref{bocor}) a smooth graph with a vertical tangency at the origin. 

Examples for variation (ii) of Theorem \ref{exprop} are then provided by taking $\psi=\psi_1\circ\psi_{3}^{-1}$. Indeed $\psi_3$  maps the nowhere-Legendrian submanifold $\Lambda=\{x_1=y_1=y_2=\cdots=y_n=0\}\subset \mathcal{N}$ homeomorphically to itself, fixing the origin (see Corollary \ref{wall}), and $\psi_1$ maps $\Lambda$ to a submanifold that is tangent to the contact distribution at the origin.  Moreover the fact that $\psi_3$ is bounded above but not below near the origin while $\psi_1$ is bounded both above and below near the origin readily implies that $\psi_1\circ\psi_{3}^{-1}$ is bounded below but not above near the origin.
\end{proof}



\begin{thebibliography}{99}
\bibitem[AM78]{AM} R. Abraham and J. Marsden. \emph{Foundations of Mechanics}. 2nd ed. Benjamin/Cummings Pub., Reading, Mass., 1978.
\bibitem[Ak01]{Ak} M. Akaho. \emph{Hofer's symplectic energy and Lagrangian intersections in contact geometry}. J. Math. Kyoto Univ. \textbf{41} (2001), no. 3, 593--609.
\bibitem[AFM15]{AFM} P. Albers, U. Fuchs, and W. Merry. \emph{Orderability and the Weinstein conjecture}. Compos. Math. \textbf{151} (2015), no. 12, 2251--2272. 
\bibitem[AH09]{AH} P. Albers and H. Hofer. \emph{On the Weinstein conjecture in higher dimensions}. Comment. Math. Helv. \textbf{84} (2009), no. 2, 429--436.
\bibitem[Ar46]{Ar} R. Arens. \emph{Topologies for homeomorphism groups}. Amer. J. Math. \textbf{68} (1946), no. 4, 593--610.
\bibitem[BO16]{BO} L. Buhovsky and E. Opshtein. \emph{Some quantitative results in $\mathcal{C}^0$ symplectic geometry}. Invent. Math. \textbf{205} (2016), 1--56.
\bibitem[CMP19]{CMP} R.  Casals, E. Murphy, and F. Presas. \emph{Geometric criteria for overtwistedness}. J. Amer. Math. Soc. \textbf{32} (2019), no. 2, 563--604.
\bibitem[CCD19]{CCD} B. Chantraine, V. Colin, and G. Dimitroglou Rizell. \emph{Positive Legendrian isotopies and Floer theory}. Ann. Inst. Fourier \textbf{69} (2019), no. 4, 1679--1737.
\bibitem[C00]{Che} Yu. Chekanov. \emph{Invariant Finsler metrics on the space of Lagrangian embeddings}. Math. Z. \textbf{234} (2000), 605--619.
\bibitem[CE12]{CE} K. Cieliebak and Y. Eliashberg. errata (2014) to \emph{From Stein to Weinstein and back. Symplectic geometry of affine complex manifolds}, AMS, Providence, 2012. Available at \href{https://www.ams.org/publications/authors/books/postpub/coll-59-errata.pdf}{https://www.ams.org/publications/authors/books/postpub/coll-59-errata.pdf}.
\bibitem[DS16]{RS} G. Dimitroglou Rizell and M. Sullivan. \emph{An energy-capacity inequality for Legendrian submanifolds}. To appear in J. Topol. Anal., arXiv:1608.06232v4.
\bibitem[EHS95]{EHS} Y. Eliashberg, H. Hofer, and D. Salamon. \emph{Lagrangian intersections in contact geometry}. Geom. Funct. Anal. \textbf{5} (1995), no. 2, 244--269.
\bibitem[Et]{Et} J. Etnyre. \emph{Legendrian and transversal knots}. Handbook of knot theory, 105--185, Elsevier B. V., Amsterdam, 2005.
\bibitem[H15]{Huang} Y. Huang. \emph{On Legendrian foliations in contact manifolds I: Singularities and neighborhood theorems}. Math. Res. Lett. \textbf{22} (2015), no. 5, 1373--1400.
\bibitem[HLS15]{HLS} V. Humili\`ere, R. Leclercq, and S. Seyfaddini. \emph{Coisotropic rigidity and $C^0$-symplectic geometry}. Duke Math. J. \textbf{164} (2015), no. 4, 767--799.
\bibitem[KM97]{KM} A. Kriegl and P. W. Michor. \emph{The convenient setting of global analysis}. Math. Surv. Mon. \textbf{53}, AMS, 1997.
\bibitem[LdL19]{LdL} M. Lainz Valc\'azar and M. de Le\'on. \emph{Contact Hamiltonian systems}. J. Math. Phys. \textbf{60} (2019), no. 10, 102902. arXiv:1811.03367v1.
\bibitem[LS94]{LS} F. Laudenbach and J.-C. Sikorav \emph{Hamiltonian disjunction and limits of Lagrangian submanifolds}. Internat. Math. Res. Notices \textbf{1994}, no. 4, 161--168.
\bibitem[Mas16]{Mas} P. Massot. \emph{Quelques applications de la convexit\'e en topologie de contact}. Habilitation, Universit\'e Paris-Sud, 2016.
\bibitem[McSa17]{MS3} D. McDuff and D. Salamon. \emph{Introduction to symplectic topology}. 3rd ed. Oxford Graduate Texts in Mathematics, 2017.
\bibitem[Mul90]{Mul} M.-P. Muller. \emph{Une structure symplectique sur $\mathbb{R}^6$ avec une sph\`ere lagrangienne plong\'ee et un champ de Liouville complet}.  Comment. Math. Helv. \textbf{65} (1990), no. 4, 623--663.
\bibitem[M\"uSp14]{MSrig} S. M\"uller and P. Spaeth. \emph{Gromov's alternative, Eliashberg's shape invariant, and $C^0$-rigidity of contact diffeomorphisms}. Int. J. Math. \textbf{25} (2014), no. 14, 13 pp.
\bibitem[M\"uSp15]{MSp} S. M\"uller and P. Spaeth. \emph{Topological contact dynamics I: symplectization and applications of the energy-capacity inequality}. Adv. Geom. \textbf{15} (2015), no. 3, 349--380.
\bibitem[M\"u19]{M19} S. M\"uller. \emph{$C^0$-characterization of symplectic and contact embeddings and Lagrangian rigidity}. Internat. J. Math. \textbf{30} (2019), no. 9, 1950035, 48 pp.
\bibitem[Mur13]{Mur} E. Murphy \emph{Closed exact Lagrangians in the symplectization of contact manifolds}. arXiv:1304.6620.
\bibitem[Oh97]{Oh} Y.-G. Oh. \emph{Gromov-Floer theory and disjunction energy of compact Lagrangian embeddings}. Math. Res. Lett. \textbf{4} (1997), 895--905.
\bibitem[RZ18]{RZ} D. Rosen and J. Zhang. \emph{Chekanov's dichotomy in contact topology}. arXiv:1808.08459v1.
\bibitem[Sh16]{She} E. Shelukhin. \emph{The Hofer norm of a contactomorphism}. J. Symplectic Geom. \textbf{15} (2017), no. 4, 1173--1208.
\bibitem[Si94]{Si} J.-C. Sikorav. \emph{Some properties of holomorphic curves in almost complex manifolds}. In \emph{Holomorphic curves in symplectic geometry} Progr. Math. \textbf{117}, Birkh\"auser, 1994, 165--189.
\bibitem[U14]{U1} M. Usher. \emph{Submanifolds and the Hofer norm}. J. Eur.  Math. Soc. \textbf{16} (2014), no. 8, 1571--1616.
\bibitem[U15]{U2} M. Usher. \emph{Observations on the Hofer distance between closed subsets}. Math. Res. Lett. \textbf{22} (2015), no. 6, 1805--1820.
\bibitem[U19]{U3} M. Usher. \emph{Local rigidity, symplectic homeomorphisms, and coisotropic submanifolds}. arXiv:1912.13043v2.

\end{thebibliography}
\end{document}